\documentclass[12pt,letterpaper,reqno]{amsart}
\usepackage{fullpage}
\usepackage{amsmath,amsthm,amssymb,amscd}
\usepackage{enumerate}
\usepackage{enumitem}
\usepackage{array}
\usepackage{float}
\usepackage{bbm}
\usepackage{bm}
\usepackage{stmaryrd}
\usepackage{mathrsfs}
\usepackage{comment}
\usepackage{mathtools}

\usepackage{hyperref}
\hypersetup{colorlinks=true,linkcolor=blue,citecolor=blue,urlcolor=blue}

\theoremstyle{plain}
\newtheorem{theorem}{Theorem}[section]
\newtheorem{proposition}[theorem]{Proposition}
\newtheorem{lemma}[theorem]{Lemma}
\newtheorem{corollary}[theorem]{Corollary}
\newtheorem{conjecture}[theorem]{Conjecture}

\mathtoolsset{showonlyrefs=true}
\numberwithin{equation}{section}
\numberwithin{figure}{section}
\numberwithin{table}{section}

\let\subsectiontemp\subsection
\renewcommand{\subsection}[1]{ 
    \subsectiontemp{#1} \hfill\vspace{0.5\linespacing} 
}

\allowdisplaybreaks[1]

\newcommand{\lrabs}[1]{\!\left\lvert #1 \right\lvert}
\newcommand{\lrp}[1]{\!\left(#1\right)}
\newcommand{\lrb}[1]{\!\left[#1\right]}

\newcommand{\ZZ}{\mathbb{Z}}

\newcommand{\new}{{\rm new}}
\newcommand{\nw}{{\rm new}}
\newcommand{\Tr}{\operatorname{Tr}}
\newcommand{\sgn}{\operatorname{sgn}}

\author[T. Nelson]{Timothy Nelson}
\address[T. Nelson]{School of Science, Technology and Health, Gordon College, Wenham, MA}

\author[E. Ross]{Erick Ross}
\address[E. Ross]{School of Mathematical and Statistical Sciences, Clemson University, Clemson, SC}
\email{erickjohnross@gmail.com}

\author[M. Wassercug]{Maya Wassercug}
\address[M. Wassercug]{Department of Mathematics, Northwestern University, Evanston, IL}

\author[H. Xue]{Hui Xue}
\address[H. Xue]{School of Mathematical and Statistical Sciences, Clemson University, Clemson, SC}
\email{huixue@clemson.edu}

\subjclass[2020]{Primary 11F25; Secondary 11F72 and 11F11.}
\keywords{Hecke operator; Hecke polynomial; Eichler-Selberg trace formula; Atkin-Lehner sign pattern.}

\title{Asymptotics of Hecke polynomial coefficients \\ on the Atkin-Lehner eigenspaces}

\begin{document}

\begin{abstract}
    Let $S_k^\sigma(N)$ denote the space of cusp forms of level $N$, weight $k$, and Atkin-Lehner sign pattern $\sigma$, and $S_k^{\mathrm{new},\sigma}$ denote its new subspace. In this paper, we study the asymptotic behavior of the coefficients of the $m$-th Hecke polynomial over $S_k^\sigma(N)$ and $S_k^{\mathrm{new},\sigma}(N)$. In particular, we show that in certain settings, all but finitely many of these coefficients take a particular sign. We also study settings in which the coefficients do not tend to any particular sign.
\end{abstract}

\maketitle

\tableofcontents


\section{Introduction}
For $N \ge 1$ and $k \ge 2$ even, let $S_k(N)$ denote the space of cusp forms of weight $k$ and level $N$. For $m \ge 1$ coprime to $N$, let $ T_m'(N,k) = m^{-(k-1)/2} \,T_m(N,k)$ denote the normalized $m$-th Hecke operator over $S_k(N)$, and write the characteristic polynomial of $T_m'(N,k)$ as 
\begin{align}
    T'_m(N,k)(x) = \sum_{0 \le r \le d} c_r(m,N,k)x^{d-r} \qquad \text{where } d=\dim S_k(N).
\end{align}
For simplicity, we will often just write $T_m'(N,k)$ as $T_m'$ and $c_r(m,N,k)$ as $c_r$. Additionally, we note that the coefficients $c^\nw_r(m,N,k)$ can be defined in the same way over the newspace $S_k^\nw(N)$. 

These coefficients $c_r(m,N,k)$ encode important information about $S_k(N)$. For example, the dimension of $S_k(N)$ is given by $-c_1(1,N,k)$, the determinant of $T_m'$ over $S_k(N)$ is given by $(-1)^d c_d(m,N,k)$, and the trace of $T_m'$ over $S_k(N)$ is given by $-c_1(m,N,k)$ (which has historically been especially important in the theory of modular forms). 
One important aspect of studying the trace is to determine whether or not it vanishes. For example, Lehmer's conjecture asserts that $c_1(m,1,12) \ne 0$ for all $m \ge 1$. More generally, Rouse \cite{rouse} studied the vanishing and non-vanishing behavior of $c_1(m,N,k)$, and obtained several interesting results for $N$ and $k$ large. He also posed the generalized Lehmer conjecture \cite[Conjecture 1.5]{rouse}: that $c_1(m,N,k) \ne 0$ for $m$ not a square and $k = 12$ or $k \ge 16$. Rouse himself proved this conjecture for $m=2$. The case when $m=3$ and $N=1$ was established in \cite{chiriac-williams}. More recently, the second coefficient has been studied extensively (including by the second and fourth authors); \cite{CLAYTON2024186}, \cite{Ross-Xue-second}, and \cite{cason-jim-medlock-ross-vilardi-xue} all studied the non-vanishing of $c_2(m,N,k)$, and \cite{CLAYTON-non-rep} studied the non-repetition of $c_2(m,N,k)$. For the case of general coefficients, \cite{ross-xue} determined the asymptotic behavior of $c_r(m,N,k)$. In particular, for even $r$, this resolved \cite[Conjecture 5.1]{CLAYTON2024186} (which predicts that $c_r(m,1,k) \ne 0$) in all but finitely many cases.

In this paper, we extend the above works to study the Hecke polynomial coefficients over the Atkin-Lehner sign-pattern subspaces $S_k^\sigma(N)$ and $S_k^{\nw,\sigma}(N)$. We first recall the definition of these subspaces in the following paragraph.

For $Q \parallel N$, let $W_Q$ denote the $Q$-th Atkin–Lehner involution on $S_k(N)$. The $W_Q$ involutions simultaneously diagonalize $S_k(N)$ with eigenvalues $\pm 1$. 
So define a \textit{sign pattern} $\sigma : \{ Q \parallel N\} \to \{\pm 1\}$ to be a multiplicative function on the exact divisors of $N$. Equivalently, $\sigma$ can be thought of as a character on the abelian group generated by the $W_Q$.
Then $S_k(N)$ can be decomposed into a direct sum of eigenspaces
\begin{align}
    S_k(N) = \bigoplus_\sigma S_k^\sigma(N), \qquad  S_k^\sigma(N) := \{f \in S_k(N) : W_{Q}f = \sigma(Q)f \text{ for each } Q \mid\mid N\}.
\end{align}
Here, the direct sum runs over all sign patterns $\sigma$ for $N$. Note that these sign pattern spaces $S_k^\sigma(N)$ are a finer decomposition of the more well-known Fricke sign spaces $S_k^\pm(N)$ (which come from only specifying the sign of $W_N$). Finally, we note that one can also define the sign pattern newspaces $S_k^{\nw,\sigma}(N)$ in an identical way.

We now outline the main results of this paper. In the following discussion, we will focus on the full space $S_k^\sigma(N)$. However, as discussed in Section \ref{sec:extending-to-newspace}, all of our main results can be extended to the newspace $S_k^{\nw,\sigma}(N)$ as well. 

Following the technique developed in \cite{ross-xue}, we first determine the asymptotic behavior of the $r$-th coefficient $c_r(m, N, k,\sigma)$ as $N+k \to \infty$ when $m$ is a square.
\begin{theorem} \label{thm:main-thm-square}
    Fix an integer $r \geq 0$ and a square $m \geq 1$. Then for $N$ coprime to $m$, sign patterns $\sigma$ for $N$, and $k \geq 2$ even, we have that
    $$c_r(m, N, k,\sigma) = \frac{(-1)^r}{r!} \left( \frac{1}{\sqrt{m}} \frac{k - 1}{12} \frac{\psi(N)}{2^{\omega(N)}} \right)^r + O_m\left(k^{r - 1}N^{r-1/2+\varepsilon}\right).$$
\end{theorem}

In the case when $m$ is not a square, we determine the asymptotic behavior of $c_r=c_r(m,N,k,\sigma)$ as $k \to \infty$. 
\begin{theorem} \label{thm:main-thm-nonsquare}
    Fix an integer $r \geq 0$, a non-square $m\geq 0$, and $N\geq 1$ coprime to $m$. Then for sign patterns $\sigma$ for $N$ and $k \geq 2$ even, we have that
    \begin{align*}
        c_{2r} &= \frac{(-1)^r}{(2r)!!} \left(\frac{\sigma_1(m)}{m} \frac{k-1}{12} \frac{\psi(N)}{2^{\omega(N)}}
        \right)^r + O_{m,N}\left(k^{r-1}\right) \quad \text{ and} \\
        c_{2r+1} &= c_1 \cdot \frac{(-1)^r}{(2r)!!} \left(\frac{\sigma_1(m)}{m} \frac{k-1}{12} \frac{\psi(N)}{2^{\omega(N)}}\right)^r + O_{m,N}\left(k^{r-1}\right).
    \end{align*}
\end{theorem}

Note that in particular, these two results determine the asymptotic sign of $c_r(m,N,k,\sigma)$ in the following settings.

\begin{corollary} \label{cor:coef-sign-square}
    Fix an integer $r \geq 0$ and a square $m \geq 1$. Then $c_r(m, N, k, \sigma)$ has sign $(-1)^r$ for all but finitely many triples $(N, k,\sigma)$.
\end{corollary}

\begin{corollary} \label{cor:even-coef-sign-nonsquare}
    Fix an integer $r \geq 0$, a non-square $m\geq 1$, and $N\geq 1$ coprime to $m$. Then for each sign pattern $\sigma$ for $N$, $c_{2r}(m,N,k,\sigma)$ has sign $(-1)^r$ for all but finitely many values of $k$.
\end{corollary}
We would like to point out here that all of our work is effective. In fact, in Propositions \ref{prop:second-coef-sign-square} and \ref{prop:second-coef-sign-nonsquare}, we will write down an explicit bound for how large $N$ and $k$ need to be in order for the $c_2(m,N,k,\sigma)$ sign predictions of Corollaries \ref{cor:coef-sign-square} and \ref{cor:even-coef-sign-nonsquare} to hold.

Now, in light of Theorem \ref{thm:main-thm-square} and \cite[Corollary 4.2]{ross-xue}, it is natural to guess that the coefficient $c_{2r}(m,N,k,\sigma)$ should have sign $(-1)^r$ asymptotically in $N$ as well as in $k$ (which would follow from an improvement of Theorem \ref{thm:main-thm-nonsquare} to also cover asymptotic behavior in $N$). Surprisingly, however, this turns out to not be true; we construct two infinite families of pairs $(N,\sigma)$ such that the corresponding second Hecke coefficients have opposing signs. This is in stark contrast to the behavior of the Hecke coefficients over the non-sign pattern spaces $S_k(N)$ and $S_k^\nw(N)$; compare with
\cite[Proposition 4.10]{CLAYTON2024186},
\cite[Theorem 1.2]{Ross-Xue-second}, \cite[Theorem 1.1]{cason-jim-medlock-ross-vilardi-xue}, and \cite[Corollary 4.2]{ross-xue}.
\begin{theorem} \label{thm:counterexample}
    Fix a non-square $m\geq1$, and an even integer $k\geq 2$. Then there exists an infinite family of pairs $(N,\sigma)$ such that $\sgn c_{2}(m,N,k,\sigma) = +1$, and an infinite family of pairs $(N,\sigma)$ such that $\sgn c_2(m,N,k,\sigma) = -1$. 
\end{theorem}

The key technical input that makes the results of this work possible is an explicit trace estimate for $\Tr_{S_k^\sigma(N)}T'_m$ and $\Tr_{S_k^{\new,\sigma}(N)}T'_m$, given in Proposition \ref{prop:final-trace-formula}.
The proof of this Proposition has been relegated to Appendix \ref{sec:trace-calculation-appendix} due to its length.

We would like to briefly mention a few reasons why we are interested to extend the study of the Hecke polynomial coefficients to the Atkin-Lehner eigenspaces in particular. The Hecke coefficients have been observed to possess many nice properties in the global space $S_k(N)$. It is interesting to determine which of these properties also hold at a finer level in the $S_k^\sigma(N)$, and which properties only appear globally. In fact, the sign pattern spaces are, in some sense, the finest possible natural decomposition on which one can study such questions; the $S_k^{\nw,\sigma}(N)$ seem to be the finest possible subspaces of $S_k^\nw(N)$ which respect the Hecke operators and Galois conjugation \cite[Section 5]{chow-ghitza}. 

We are also interested to extend this study to the sign pattern spaces because it involves an explicit trace formula for $T_m'$ over $S_k^\sigma(N)$, namely Proposition \ref{prop:final-trace-formula}. We believe that this trace formula will be useful to others as well; such a result has many different applications beyond just determining the behavior of the Hecke polynomial coefficients. For example, such an explicit trace formula is needed to prove effective equidistribution of Hecke eigenvalues along the lines of \cite{Murty-Sinha} for $S_k^\sigma(N)$. This will be carried out in the forthcoming paper \cite{kumar-mondal-ross-xue}. 
Very recently, Padurariu-Park-Voight \cite{Padurariu-Park-Voight} computed the complete list of pairs $(N,\sigma)$ (with $N$ squarefree) such that $\dim S_2^{\nw,\sigma}(N) = 0$. Our explicit trace formula would also be useful to extend this classification to general weight $k$.

Finally, we note an open problem related to the results of this paper. We conjecture that the second Hecke coefficients are always non-vanishing over $S_k^\sigma(N)$ and $S_k^{\nw,\sigma}(N)$ for weight $k \ge 8$. 
\begin{conjecture} \label{conj:nonvanishing}
    Let $m \ge 1$, $N \ge 1$ be coprime to $m$, and $\sigma$ be a sign pattern for $N$.
    Then $c_2(m,N,k,\sigma) \ne 0$ for all even integers $k \ge 8$ with $\dim S_k^\sigma(N) \ge 2$, and $c_2^\nw(m,N,k,\sigma) \ne 0$ for all even integers $k \ge 8$ with $\dim S_k^{\nw,\sigma}(N) \ge 2$.
\end{conjecture}
We have verified this conjecture for all $8 \le k \le 40$, $1 \le m \le 50$, $1 \le N \le 300$, and sign patterns $\sigma$ for $N$; see \cite{ross-code} for the code.

We now give an outline of the paper. 
In Section \ref{sec:newton-girard-identities}, we use the Newton–Girard identities to relate the coefficients $c_r(m,N,k,\sigma)$ to the eigenvalues of $T'_m$ over $S_k^\sigma(N)$.
In Sections \ref{sec:main-theorem-square} and \ref{sec:main-theorem-nonsquare}, we use the trace formula estimate from Appendix \ref{sec:trace-calculation-appendix} to determine the asymptotic behavior of $c_r(m,N,k,\sigma)$, proving Theorems \ref{thm:main-thm-square} and \ref{thm:main-thm-nonsquare}, respectively.
Next, in Section \ref{sec:counterexample} we prove Theorem \ref{thm:counterexample} by constructing two infinite families of pairs $(N,\sigma)$ with second Hecke coefficients of opposing signs.
In Section \ref{sec:extending-to-newspace}, we extend all of our main results to the newspace, and provide explicit estimates for $c_r^\new(m,N,k,\sigma)$.
Lastly, in Section \ref{sec:sign-of-second-coeff} we prove Propositions \ref{prop:second-coef-sign-square} and \ref{prop:second-coef-sign-nonsquare}, which give explicit bounds related to our sign prediction for $c_2(m,N,k,\sigma)$. 
We also include three appendices: Appendix \ref{sec:trace-calculation-appendix} computes an explicit trace formula for $T_m'$ over $S_k^\sigma(N)$, Appendix \ref{sec:trace-formula-prime} specializes this trace formula to the case of prime level, and Appendix \ref{sec:class_number_lemma} estimates the asymptotic behavior of a certain family of class numbers.


\section{Newton-Girard identities} \label{sec:newton-girard-identities}

In this section, we use the Newton-Girard identities to relate the coefficients $c_r=c_r(m,N,k,\sigma)$ to the traces of certain Hecke operators.

Let $d=\dim S_k^\sigma(N)$, and $\lambda_1,\dots,\lambda_d$ be the eigenvalues of $T_m'$ over $S_k^\sigma(N)$. Note that $c_r$ can be expressed in terms of the $r$-th elementary symmetric polynomial in these eigenvalues, i.e.
\[c_r=(-1)^r\sum_{1\leq i_1<\cdots <i_r\leq d}\lambda_{i_1}\lambda_{i_2}\cdots \lambda_{i_r} .\]
Let $p_r$ denote the $r$-th moment of these eigenvalues,
\begin{equation} \label{pj-definition}
    p_r  :=  \sum_{i=1}^d \lambda_i^r.
\end{equation}

Importantly, $p_r$ can be expressed in terms of the traces of Hecke operators. For instance, in the case of the second coefficient, we have that 
\begin{equation}\label{eqn:second-coefficient-formula}  
  c_2 = \sum_{1\leq i<j\leq d}\lambda_{i}\lambda_{j} = \frac{1}{2}\Big(p_1^2 - p_2\Big) 
    = \frac{1}{2}\left(\left(\Tr_{S_k^\sigma(N)} T'_m\right)^2 - \Tr_{S_k^\sigma(N)} T'^{\,2}_m\right).
\end{equation}
In fact, this can be generalized to the following.

\begin{lemma}[Newton-Girard Identities] \label{lem:newton-girard-id}
    Let $c_r$ and $p_r$ be defined as above. Then for $r\geq1$, 
    $$c_r = \frac{-1}{r} \sum_{j=1}^r c_{r-j}p_j.$$
\end{lemma}

These identities give a recursive formula for the coefficients $c_r$ in terms of the moments $p_r$. And the moments $p_r$ can be calculated via the explicit trace formula for $T_m'$, computed in Appendix \ref{sec:trace-calculation-appendix}.



\section{Asymptotics of coefficients for square \texorpdfstring{$m$}{m}} \label{sec:main-theorem-square}

In the following two sections we utilize the trace formula estimate Proposition  \ref{prop:final-trace-formula} to show Theorems \ref{thm:main-thm-square} and
\ref{thm:main-thm-nonsquare}. The argument for both sections follows the same strategy laid out in \cite{ross-xue}. In this section, we address the case where $m$ is a square, and in the next section, we will address the case where $m$ is not a square. 

As usual, in the following, $\sigma_t(n):=\sum_{d\mid n}d^{\,t}$,
$\omega(n)$ denotes the number of prime divisors of $n$,
and $\psi(n) := n\prod_{p \mid n} \lrp{1 + \frac{1}{p}}$ denotes the Dedekind psi function. Note that $\sigma_0$ and $\sigma_1$ are not to be confused with the sign pattern $\sigma$.
Additionally, we will use the notation $f(N) = O(N^\varepsilon)$ to mean that $f(N) = O_\varepsilon(N^\varepsilon)$ for all $\varepsilon > 0$.

First, we give the following estimate for the $p_j$, as defined in Section \ref{sec:newton-girard-identities}. 
\begin{lemma} \label{lem:pj-bound-square}
    Fix an integer $r \geq 1$ and a square $m \geq 1$. Then for $N$ coprime to $m$, sign patterns $\sigma$ for $N$, and $k \geq 2$ even,
    \begin{align*}
        p_1 &= \frac{1}{\sqrt{m}} \frac{k - 1}{12} \frac{\psi(N)}{2^{\omega(N)}} + O_{r,m}\left(N^{1/2 + \varepsilon}\right), \text{ and} \\
        p_j &= O_{r,m}\left(kN\right) \qquad\qquad\,\,\,\, \text{ for all } 1 \leq j \leq r.
    \end{align*}
\end{lemma}

\begin{proof}
    The first claim follows from Proposition  \ref{prop:final-trace-formula}.
    For the second claim, observe that by Proposition  \ref{prop:final-trace-formula} and the fact that $\frac{\psi(N)}{2^{\omega(N)}} \le N$, we have
    $$d  :=  \dim S_k^\sigma (N) = \operatorname{Tr}_{S_k^\sigma(N)}T'_1 = O_{r,m}\left(kN\right).$$
    By Deligne's bound $|\lambda_i|\leq \sigma_0(m)$, this then implies that
    $$|p_j| = \lrabs{\sum_{i = 1}^d \lambda_i^j} \leq \sum_{i = 1}^d \sigma_0(m)^j \le \sigma_0(m)^j \cdot d  =  O_{r,m}\left(kN\right),$$
    completing the proof.
\end{proof}

We now determine the asymptotic behavior of $c_r(m,N,k,\sigma)$ as $N+k\to \infty$.

{
\renewcommand{\thetheorem}{\ref{thm:main-thm-square}}
\begin{theorem}
    Fix an integer $r \geq 0$ and a square $m \geq 1$. Then for $N$ coprime to $m$, sign patterns $\sigma$ for $N$, and $k \geq 2$ even, we have that
    $$c_r(m, N, k,\sigma) = \frac{(-1)^r}{r!} \left( \frac{1}{\sqrt{m}} \frac{k - 1}{12} \frac{\psi(N)}{2^{\omega(N)}} \right)^r + O_{r,m}\left(k^{r - 1}N^{r-1/2+\varepsilon}\right).$$
\end{theorem}
\addtocounter{theorem}{-1}
}

\begin{proof}
    We proceed via strong induction on $r$. The base case of $r = 0$ is immediate since $c_0 = 1$.
    
    For $r \geq 1$, we have by Lemma \ref{lem:newton-girard-id} that
    \begin{equation} \label{eqn:square-m-c_r-summation}
        c_r = \frac{-1}{r} \sum_{j=1}^r c_{r-j}p_j = \frac{-1}{r} \left[ c_{r - 1}p_1 + \sum_{j=2}^r c_{r-j}p_j \right].
    \end{equation}
    Then by the induction hypothesis,
    \begin{align*}
        c_{r-1} &= \frac{(-1)^{r-1}}{(r-1)!} \left( \frac{1}{\sqrt{m}} \frac{k - 1}{12} \frac{\psi(N)}{2^{\omega(N)}} \right)^{r-1} + O_{r,m}\!\left(k^{r-2}N^{r - 3/2 + \varepsilon}\right), \\
        c_{r-j} &= O_{r,m}\!\left(k^{r-2}N^{r-2}\right) \qquad  \text{ for } 2 \leq j \leq r,
    \end{align*}
    and by Lemma \ref{lem:pj-bound-square},
    \begin{align*}
        p_1 &= \frac{1}{\sqrt{m}} \frac{k - 1}{12} \frac{\psi(N)}{2^{\omega(N)}} + O_{r,m}\!\left(N^{1/2+\varepsilon}\right), \\
        p_j &= O_{r,m}\left(kN\right) \qquad \text{ for } 2 \leq j \leq r.
    \end{align*}
    
    Applying these estimates to \eqref{eqn:square-m-c_r-summation}, we obtain
    \begin{align*}
        c_r &= \frac{-1}{r} \left[ c_{r - 1}p_1 + \sum_{j=2}^r c_{r-j}p_j \right] \\
        &= \frac{-1}{r} \left[ \left( \frac{(-1)^{r-1}}{(r - 1)!} \left( \frac{1}{\sqrt{m}} \frac{k - 1}{12} \frac{\psi(N)}{2^{\omega(N)}} \right)^{r-1} + O_{r,m}\left(k^{r-2}N^{r-3/2+\varepsilon}\right) \right) \right. \\ 
        &\qquad\quad\ \left. \times \left( \frac{1}{\sqrt{m}} \frac{k - 1}{12} \frac{\psi(N)}{2^{\omega(N)}} + O_{r,m}\left(N^{1/2 + \varepsilon}\right) \right) + \sum_{j=2}^r O_{r,m}\left(k^{r-2}N^{r-2}\right) \cdot O_{r,m}\left(kN\right) \right] \\
        &= \frac{-1}{r} \left[ \frac{(-1)^{r-1}}{(r - 1)!} \left( \frac{1}{\sqrt{m}} \frac{k - 1}{12} \frac{\psi(N)}{2^{\omega(N)}} \right)^r + O_{r,m}\left(k^{r-1}N^{r-1/2+\varepsilon}\right) \right] \\
        &= \frac{(-1)^r}{r!} \left( \frac{1}{\sqrt{m}} \frac{k - 1}{12} \frac{\psi(N)}{2^{\omega(N)}} \right)^r + O_{r,m}\left(k^{r-1}N^{r-1/2+\varepsilon}\right),
    \end{align*}
    as desired.
\end{proof}

This theorem allows us to determine the sign of $c_r(m, N, k,\sigma)$ in all but finitely many cases.

{
\renewcommand{\thetheorem}{\ref{cor:coef-sign-square}}
\begin{corollary}
    Fix an integer $r \geq 0$ and a square $m \geq 1$. Then $c_r(m, N, k, \sigma)$ has sign $(-1)^r$ for all but finitely many triples $(N, k,\sigma)$.
\end{corollary}
\addtocounter{theorem}{-1}
}

This corollary follows immediately from the fact that $\frac{\psi(N)}{2^{\omega(N)}} \gg N^{1-\varepsilon}$ \cite[Theorem 2.10]{montgomery-vaughn}.


\section{Asymptotics of coefficients for non-square \texorpdfstring{$m$}{m}} \label{sec:main-theorem-nonsquare}

In this section, we address the case when $m$ is not a square. We will only give asymptotics for $c_r(m,N,k,\sigma)$ as $k \to \infty$ (as opposed to $N+k\to\infty$ from the previous section). In fact, this is necessary; in Section \ref{sec:counterexample}, we will show that the corresponding result does not hold asymptotically in $N$.

 First, we give the following estimates on $p_j$. 
\begin{lemma} \label{lem:pj-bound-nonsquare}
    Fix an integer $r \geq 0$, a non-square $m \geq 1$, and an integer $N\geq 1$ coprime to $m$. Then for sign patterns $\sigma$ for $N$ and $k\geq 2$ even, we have that
    \begin{align*}
        p_1 &= -c_1 = O_{r,m,N}\!\left(1\right), \\
        p_2 &= \frac{\sigma_1(m)}{m} \frac{k-1}{12} \frac{\psi(N)}{2^{\omega(N)}} + O_{r,m,N}\!\left(1\right), \\
        p_3 &= O_{r,m,N}\!\left(1\right), \\
        p_j &= O_{r,m,N}\!\left(k\right) \quad \text{ for all } 1 \leq j \leq r.
    \end{align*}
\end{lemma}

\begin{proof}
    The first claim follows immediately from Proposition  \ref{prop:final-trace-formula}. For the second claim, note that $p_2 = \operatorname{Tr}_{S_k^\sigma(N)}{T'^{\,2}_{m}}$. Then by the Hecke operator composition formula \cite[Theorem 10.2.9]{cohen2017modular} and Proposition  \ref{prop:final-trace-formula},
    \begin{align*}
        p_2 = \operatorname{Tr}_{S_k^\sigma(N)}{T'^{\,2}_{m}} = \sum_{d|m} \operatorname{Tr}_{S_k^\sigma(N)}T'_{m^2/d^2} &= \sum_{d|m}\lrp {\frac{d}{m} \frac{k-1}{12} \frac{\psi(N)}{2^{\omega(N)}} + E_{k,\sigma}\lrp{\frac{m^2}{d^2},N}} \\
        &= \frac{\sigma_1(m)}{m} \frac{k-1}{12} \frac{\psi(N)}{2^{\omega(N)}} +\sum_{d\mid m} E_{k,\sigma}\lrp{\frac{m^2}{d^2},N}. \\
        &= \frac{\sigma_1(m)}{m} \frac{k-1}{12} \frac{\psi(N)}{2^{\omega(N)}} + O_{r,m,N}\!\left(1\right).
    \end{align*}
    For the third claim, we have by the Hecke operator composition formula and Proposition  \ref{prop:final-trace-formula},
    \begin{align*}
        p_3 = \operatorname{Tr}_{S_k^\sigma(N)}T_{m}'^{\,3} = \operatorname{Tr}_{S_k^\sigma(N)} \sum_{d|m} T'_{m^2/d^2} T'_m &= \sum_{d|m} \sum_{\delta | (m^2/d^2,m)} \operatorname{Tr}_{S_k^\sigma(N)} T'_{m^3/d^2\delta^2} \\
        &= \sum_{d|m} \sum_{\delta | (m^2/d^2,m)} E_{k,\sigma}\lrp{\frac{m^3}{d^2\delta^2},N} \\
        &= \sum_{d|m} \sum_{\delta | (m^2/d^2,m)} O_{r,m,N}\!\left(1\right) \\
        &= O_{r,m,N}\!\left(1\right).
    \end{align*}
    Lastly, the fourth claim follows from an identical argument as in Lemma \ref{lem:pj-bound-square}.
\end{proof}

For fixed $m$, $r$, and $N$, we now determine the asymptotic behavior of $c_r(m, N, k,\sigma)$ as $ k \to \infty$.

{
\renewcommand{\thetheorem}{\ref{thm:main-thm-nonsquare}}
\begin{theorem}
    Fix an integer $r \geq 0$, a non-square $m\geq 1$, and an integer $N\geq 1$ coprime to $m$. Then for sign patterns $\sigma$ for $N$ and $k \geq 2$ even, we have that
    \begin{align*}
        c_{2r} &= \frac{(-1)^r}{(2r)!!} \left(\frac{\sigma_1(m)}{m} \frac{k-1}{12} \frac{\psi(N)}{2^{\omega(N)}}
        \right)^r + O_{r,m,N}\!\left(k^{r-1}\right) \quad \text{ and} \\
        c_{2r+1} &= c_1 \cdot \frac{(-1)^r}{(2r)!!} \left(\frac{\sigma_1(m)}{m} \frac{k-1}{12} \frac{\psi(N)}{2^{\omega(N)}}\right)^r + O_{r,m,N}\!\left(k^{r-1}\right)\!.
    \end{align*}
\end{theorem}
\addtocounter{theorem}{-1}
}

\begin{proof}
    We proceed via strong induction on $r$. The base case of $r = 0$ is immediate since $c_0 = 1$ and $c_1 = c_1$.
    For $r \geq 1$, we have from Lemma \ref{lem:newton-girard-id} that
    \begin{equation}\label{eqn:even-coeff-formula}
        c_{2r} = \frac{-1}{2r} \sum_{j=1}^{2r} c_{2r-j}p_j = \frac{-1}{2r} \left[ c_{2r-1}p_1 + c_{2r-2}p_2 + \sum_{j=3}^{2r} c_{2r-j}p_j \right].
    \end{equation}
    By the induction hypotheses,
    \begin{align*}
        c_{2r-1} &= O_{r,m,N}\left(k^{r-1}\right), \\
        c_{2r-2} &= \frac{(-1)^{r-1}}{(2r-2)!!}\lrp{\frac{\sigma_1(m)}{m} \frac{k-1}{12} \frac{\psi(N)}{2^{\omega(N)}}}^{r-1}
        + O_{r,m,N}\left(k^{r-2}\right), \\
        c_{2r-j} &= O_{r,m,N}\!\left(k^{r-2}\right) \qquad \text{ for } 3 \leq j \leq 2r,
    \end{align*}
    and by Lemma \ref{lem:pj-bound-nonsquare},
    \begin{align*}
        p_1 &= O_{r,m,N}\!\left(1\right), \\
        p_2 &= \frac{\sigma_1(m)}{m} \frac{k-1}{12} \frac{\psi(N)}{2^{\omega(N)}} + O_{r,m,N}\!\left(1\right), \\
        p_j &= O_{r,m,N}\!\left(k\right) \qquad \text{ for } 3 \leq j \leq 2r.
    \end{align*}
    Applying these estimates to \eqref{eqn:even-coeff-formula}, we obtain
    \begin{align*}
        c_{2r} &= \frac{-1}{2r}\left[ c_{2r-1}p_1 + c_{2r-2}p_2 + \sum_{j=3}^{2r} c_{2r-j}p_j \right] \\
        &= \frac{-1}{2r} \left[ O_{r,m,N}\!\left(k^{r-1}\right) \cdot O_{r,m,N}\!\left(1\right) + \left( \frac{(-1)^{r-1}}{(2r-2)!!} \left(\frac{\sigma_1(m)}{m} \frac{k-1}{12} \frac{\psi(N)}{2^{\omega(N)}}\right)^{r-1} + O_{r,m,N}\!\left(k^{r-2}\right) \right) \right. \\ 
        &\qquad\quad\ \left. \times \left(\frac{\sigma_1(m)}{m} \frac{k-1}{12} \frac{\psi(N)}{2^{\omega(N)}}+ O_{r,m,N}\!\left(1\right) \right) + \sum_{j=3}^{2r} O_{r,m,N}\!\left(k^{r-2}\right) \cdot O_{r,m,N}\!\left(k\right) \right] \\
        &= \frac{-1}{2r} \left[ O_{r,m,N}\!\left(k^{r-1}\right) + \frac{(-1)^{r-1}}{(2r-2)!!} \left(\frac{\sigma_1(m)}{m} \frac{k-1}{12} \frac{\psi(N)}{2^{\omega(N)}}\right)^r + O_{r,m,N}\!\left(k^{r-1}\right) + O_{r,m,N}\!\left(k^{r-1}\right) \right] \\
        &= \frac{(-1)^r}{(2r)!!} \left(\frac{\sigma_1(m)}{m} \frac{k-1}{12} \frac{\psi(N)}{2^{\omega(N)}}\right)^r + O_{r,m,N}\!\left(k^{r-1}\right),
    \end{align*}
    confirming the first claim of the inductive step.
    For the second claim of the inductive step, Lemma \ref{lem:newton-girard-id} yields that
    \begin{equation}\label{eqn:odd-coeff-formula}
        c_{2r+1} = \frac{-1}{2r+1} \sum_{j=1}^{2r+1} c_{2r+1-j}p_j = \frac{-1}{2r+1} \left[c_{2r}p_1 + c_{2r-1}p_2 + c_{2r-2}p_3 + \sum_{j=4}^{2r+1} c_{2r+1-j}p_j \right].
    \end{equation}
    Now by the induction hypothesis and the proof for $c_{2r}$, 
    \begin{align*}
        c_{2r} &= \frac{(-1)^r}{(2r)!!} \left(\frac{\sigma_1(m)}{m} \frac{k-1}{12} \frac{\psi(N)}{2^{\omega(N)}}\right)^r + O_{r,m,N}\!\left(k^{r-1}\right), \\
        c_{2r-1} &= c_1 \cdot \frac{(-1)^{r-1}}{(2r-2)!!}\left(\frac{\sigma_1(m)}{m} \frac{k-1}{12} \frac{\psi(N)}{2^{\omega(N)}}\right)^{r-1} + O_{r,m,N}\!\left(k^{r-2}\right), \\
        c_{2r-2} &= O_{r,m,N}\!\left(k^{r-1}\right), \\
        c_{2r+1-j} &= O_{r,m,N}\!\left(k^{r-2}\right) \qquad \text{ for } 4 \leq j \leq 2r + 1,
    \end{align*}
    and by Lemma \ref{lem:pj-bound-nonsquare},
    \begin{align*}
        p_1 &= -c_1, \\
        p_2 &= \frac{\sigma_1(m)}{m} \frac{k-1}{12} \frac{\psi(N)}{2^{\omega(N)}} + O_{r,m,N}(1), \\
        p_3 &= O_{r,m,N}(1), \\
        p_j &= O_{r,m,N}(k) \qquad\text{ for } 4 \leq j \leq 2r+1.
    \end{align*}
    Applying these estimates to \eqref{eqn:odd-coeff-formula}, we obtain
    \begin{align*}
        c_{2r+1} 
        &= \frac{-1}{2r+1} \left[ c_{2r}p_1 + c_{2r-1}p_2 + c_{2r-2}p_3 + \sum_{j=4}^{2r+1} c_{2r+1-j}p_j \right] \\
        &= \frac{-1}{2r+1} \Bigg[ \left( \frac{(-1)^r}{(2r)!!} \left(\frac{\sigma_1(m)}{m} \frac{k-1}{12} \frac{\psi(N)}{2^{\omega(N)}}\right)^r + O_{r,m,N}\!\left(k^{r-1}\right) \right) \cdot (-c_1)  \\ 
        &\qquad\qquad\quad  + \left( c_1 \cdot \frac{(-1)^{r-1}}{(2r-2)!!} \left( \frac{\sigma_1(m)}{m} \frac{k-1}{12} \frac{\psi(N)}{2^{\omega(N)}}\right)^{r-1} + O_{r,m,N}\!\left(k^{r-2}\right) \right) \\
        &\qquad \qquad\quad  \times \left( \frac{\sigma_1(m)}{m} \frac{k-1}{12} \frac{\psi(N)}{2^{\omega(N)}} + O_{r,m,N}\!\left(1\right) \right) + O_{r,m,N}\!\left(k^{r-1}\right) \cdot O_{r,m,N}\!\left(1\right)  \\
        &\qquad\qquad\quad  + \sum_{j=4}^{2r+1} O_{r,m,N}\!\left(k^{r-2}\right) \cdot O_{r,m,N}\!\left(k\right) \Bigg] \\
        &= \frac{-1}{2r+1} \left[ -c_1 \cdot \frac{(-1)^r}{(2r)!!} \left(\frac{\sigma_1(m)}{m} \frac{k-1}{12} \frac{\psi(N)}{2^{\omega(N)}}\right)^r + O_{r,m,N}\!\left(k^{r-1}\right)  \right. \\
        &\qquad \qquad\quad+ c_1 \cdot \frac{(-1)^{r-1}}{(2r-2)!!}\left(\frac{\sigma_1(m)}{m} \frac{k-1}{12} \frac{\psi(N)}{2^{\omega(N)}}\right)^r + O_{r,m,M}\!\left(k^{r-1}\right)  \\
        &\qquad\qquad\quad+ O_{r,m,N}\!\left(k^{r-1}\right) + O_{r,m,N}\!\left(k^{r-1}\right) \bigg{]} \\
        &= c_1 \cdot \frac{1}{2r+1} \left( \frac{(-1)^r}{(2r)!!} - \frac{(-1)^{r-1}}{(2r-2)!!} \right) \cdot \left(\frac{\sigma_1(m)}{m} \frac{k-1}{12} \frac{\psi(N)}{2^{\omega(N)}}\right)^r + O_{r,m,N}\!\left(k^{r-1}\right) \\
        &= c_1 \cdot \frac{(-1)^r}{(2r)!!} \left(\frac{\sigma_1(m)}{m} \frac{k-1}{12} \frac{\psi(N)}{2^{\omega(N)}}\right)^r + O_{r,m,N}\!\left(k^{r-1}\right),
    \end{align*}
    confirming the second claim of the inductive step.
    This completes the proof.
\end{proof}

This theorem allows us to determine the sign of $c_{2r}(m, N, k,\sigma)$ in all but finitely many cases.
{
\renewcommand{\thetheorem}{\ref{cor:even-coef-sign-nonsquare}}
\begin{corollary} 
    Fix an integer $r \geq 0$, a non-square $m\geq 1$, and $N\geq 1$ coprime to $m$. Then for each sign pattern $\sigma$ for $N$, $c_{2r}(m,N,k,\sigma)$ has sign $(-1)^r$ for all but finitely many values of $k$.
\end{corollary}
\addtocounter{theorem}{-1}
}

In contrast to the even-indexed coefficients, the behavior of the odd-indexed coefficients is determined by the behavior of the trace $- c_1$. Hence in order to obtain an analogous result for the odd-indexed coefficients $c_{2r+1}$, one would need to bound $\Tr_{S^\sigma_k(N)}T_m'$ away from $0$.


\section{Infinite families of coefficients with opposing signs} \label{sec:counterexample}
In Corollary \ref{cor:even-coef-sign-nonsquare}, we showed that $c_{2r}$ has sign $(-1)^r$ asymptotically as $k \to \infty$. In this section, we show that the analogous result does not hold asymptotically in $N$. 
{
\renewcommand{\thetheorem}{\ref{thm:counterexample}}
\begin{theorem} 
    Fix a non-square $m\geq1$, and an even integer $k\geq 2$. Then there exists an infinite family of pairs $(N,\sigma)$ such that $\sgn c_{2}(m,N,k,\sigma) = +1$, and an infinite family of pairs $(N,\sigma)$ such that $\sgn c_2(m,N,k,\sigma) = -1$. 
\end{theorem}
\addtocounter{theorem}{-1}
}

\begin{proof}
    We restrict to the case where $N=p$ is prime. Then by Proposition \ref{prop:prime-trace-formula-appendix} and the fact that the Hurwitz class number $H(D)=O(D^{1/2+\varepsilon})$, we have 
    \begin{align*}
        \Tr_{S_k^\pm(p)}T_m'&=\pm \frac{(-1)^{k/2}}{4\sqrt{m}}H(4mp)+O_m(1), \\
        \Tr_{S_k^\pm(p)}T_m'^{\,2}
        &=\sum_{d\mid m}\Tr_{S_k^\pm(p)}T_{m^2/d^2}' \\
        &= \sum_{d\mid m} \lrp{\frac{d}{m}\frac{k-1}{24}(p+1)\pm \frac{(-1)^{k/2}d}{4m}H\lrp{4\frac{m^2}{d^2}p}+O_m(1)} \\
        &=\frac{\sigma_1(m)}{m}\frac{k-1}{24}(p+1)+O_m(p^{1/2+\varepsilon}).
    \end{align*}
    Then by \eqref{eqn:second-coefficient-formula}, 
    \begin{align*}
        c_2
        &= \frac{1}{2}\left(\left(\Tr_{S_k^\pm(N)} T'_m\right)^2 - \Tr_{S_k^\pm(N)} T'^{\,2}_m\right) \\
        &=\frac{1}{32m}H(4mp)^2-\frac{\sigma_1(m)}{m}\frac{k-1}{48}(p+1)+O_m(p^{1/2+\varepsilon}).
    \end{align*}
    Observe that this estimate for $c_2$ has two main terms; one proportional to $H(4mp)^2$, and one proportional to $-(p+1)$. 
    By Proposition \ref{prop:counterexample-prop}, there exists an infinite family of primes $p$ such that $H(4mp)^2 \gg_{m} p \lrp{\log \log p}^{2}$, so $c_2$ has sign $+1$ for sufficiently large $p$ in this family. 
    Similarly, there exists another infinite family of primes $p$ such that $H(4mp)^2 \ll_{m} p\,\lrp{\log \log p}^{-2}$, so $c_2$ has sign $-1$ for sufficiently large $p$ in this other family.
\end{proof}


\section{An extension to the newspace} \label{sec:extending-to-newspace}
We now extend our results to the Hecke polynomial over the newspace $S_{k}^{\new, \sigma}(N)$. Let $T_m'^{\,\new}$ denote the restriction of $T'_m$ to  $S_{k}^{\new, \sigma}(N)$, and recall that $c_r^\new (m,N,k,\sigma)$ denotes the $r$-th coefficient of the characteristic polynomial of $T_m'^{\,\new}$, given by 
\[T_m'^{\,\new}(N,k,\sigma)(x) = \sum_{r=0}^d c_r^\new(m,N,k,\sigma)x^{d-r},\]
where $d:=\dim S_k^{\new, \sigma}(N)$.

In this section, we only state our results for admissible sign patterns, where an admissible sign pattern is defined to be any sign pattern $\sigma$ for $N$ except for the cases where $4 \parallel N$ and $\sigma(4)=+1$. This condition is not restrictive; in fact, if $\sigma$ is an inadmissible sign pattern, there are no newforms with sign pattern $\sigma$ \cite{atkin-lehner}. In contrast, \cite{ross-vanlidth-wolf-xue} found that all admissible sign patterns are close to being equally proportioned.

Given an admissible sign pattern $\sigma$, we want to determine the behavior of $c_r^\new(m,N,k,\sigma)$ as $N+k\to \infty$. Although $c_r^\new (m,N,k,\sigma)$ is not technically defined when $\dim S_k^{\new,\sigma}(N)<r$, there are only finitely many such pairs $(N,k)$ (by \cite[Theorem 1.1]{ross-vanlidth-wolf-xue} or Proposition  \ref{prop:final-trace-formula} at $m=1$). Note that this property only holds because we have restricted to the case of admissible sign patterns.

One can then extend all of the main results of this paper (namely Theorems \ref{thm:main-thm-square} and \ref{thm:main-thm-nonsquare}, Corollaries \ref{cor:coef-sign-square} and \ref{cor:even-coef-sign-nonsquare}, and Theorem \ref{thm:counterexample}) to the newspace. We omit the proofs of these results since they all follow by identical arguments as before.

We now state the analogous results. In the following, $\psi^\new(N)$ denotes 
the multiplicative function defined in \eqref{eq:psi-new-def}.

\begin{theorem}[{cf. Theorem \ref{thm:main-thm-square}}] \label{thm:main-thm-square-newspace}
    Fix an integer $r\geq0$ and a square $m\geq 1$. Then for $N\geq 1$ coprime to $m$,  admissible sign patterns $\sigma$ for $N$, and $k\geq2$ even, we have that
    \begin{multline*}
      c_r^{\new} (m,N,k,\sigma)=\frac{(-1)^r}{r!}\lrp{\frac{1}{\sqrt{m}}\frac{k-1}{12}\frac{\psi^\new(N)}{2^{\omega(N)}}
        \prod_{p^{\ell}||N}\lrp{1+\sigma(p^\ell)\frac{-\mathbbm{1}_{\ell=2}}{p^2-p-1} }
        }^r \\
        +O_{r,m}(k^{r-1}N^{r-1/2+\varepsilon}) .
    \end{multline*}
\end{theorem}

\begin{theorem}[{cf. Theorem \ref{thm:main-thm-nonsquare}}]
    \label{thm:main-thm-nonsquare-newspace}
    Fix an integer $r\geq 0$, a square $m\geq1$, and an integer $N\geq 1$ coprime to $m$. Then for admissible sign patterns $\sigma$ for $N$ and $k\geq 2$ even, we have that 
    \begin{multline*}
        c_{2r}^\new (m,N,k,\sigma)=\frac{(-1)^r}{(2r)!!}\lrp{\frac{\sigma_1(m)}{m}\frac{k-1}{12}\frac{\psi^\new(N)}{2^{\omega(N)}}
        \prod_{p^{\ell}||N}\lrp{1+\sigma(p^\ell)\frac{-\mathbbm{1}_{\ell=2}}{p^2-p-1} }}^r  \\
        +O_{r,m,N}(k^{r-1}), 
    \end{multline*}
    \begin{multline*}
        c_{2r+1}^\new (m,N,k,\sigma)=c_1^\new 
      \cdot  \frac{(-1)^r}{(2r)!!}\lrp{\frac{\sigma_1(m)}{m}\frac{k-1}{12}\frac{\psi^\new(N)}{2^{\omega(N)}}
        \prod_{p^{\ell}||N}\lrp{1+\sigma(p^\ell)\frac{-\mathbbm{1}_{\ell=2}}{p^2-p-1} }}^r \\
        +O_{r,m,N}(k^{r-1}) .
    \end{multline*}
\end{theorem}

\begin{corollary}[{cf. Corollary \ref{cor:coef-sign-square}}]
    Fix an integer $r\geq 0$ and a square $m\geq 1$. Then $c_{r}^\new(m,N,k,\sigma)$ has sign $(-1)^r$ for all but finitely many triples $(N,k,\sigma)$, where $\sigma$ denotes an admissible sign pattern for $N$. 
\end{corollary}

\begin{corollary}[{cf. Corollary \ref{cor:even-coef-sign-nonsquare}}] 
\label{cor:asymptotic-sign-nonsquare-newspace}
    Fix an integer $r\geq 0$, a non-square $m\geq 1$, and an integer $N\geq 1$ coprime to $m$. Then for admissible sign patterns $\sigma$ for $N$, $c_{2r}^\new(m,N,k,\sigma)$ has sign $(-1)^r$ for all but finitely many weights $k$.
\end{corollary}

\begin{theorem}[{cf. Theorem \ref{thm:counterexample}}]
\label{thm:counterexample-newspace}
    Fix non-square $m\geq1$, and an even integer $k\geq 2$. Then there exists an infinite family of pairs $(N,\sigma)$ such that $\sgn c^\new_{2}(m,N,k,\sigma) = +1$, and an infinite family of pairs $(N,\sigma)$ such that $\sgn c^\new_2(m,N,k,\sigma) = -1$. 
\end{theorem}


\section{Sign of the second coefficient} \label{sec:sign-of-second-coeff}

In this section, we use the explicit bounds on the error terms of the trace formula Proposition  \ref{prop:final-trace-formula} to determine how large $N$ and $k$ need to be in order for the $c_2(m,N,k,\sigma)$ sign predictions of Corollaries \ref{cor:coef-sign-square} and \ref{cor:even-coef-sign-nonsquare} to hold.

First, we consider the case when $m$ is a square.
\begin{proposition} \label{prop:second-coef-sign-square}
    Let $m$ be a square. Then $c_2(m,N,k,\sigma)$ has sign $+1$ for all $N,k,\sigma$ satisfying
    \begin{equation} \label{eq:second-coef-square-sign-condition}
        \frac{k-1}{12} \frac{\psi(N)}{2^{\omega(N)} \sqrt{N} \log(4N) \sigma_0(N)^2}  > 6.22 \cdot 10^{10}\, m^5 \sigma_0(m)^2.
    \end{equation}
\end{proposition}

\begin{proof}
    First note that by Deligne's bound, we have that 
    \begin{align*}
        \lrabs{\Tr_{S_k^\sigma(N)} T'^{\,2}_m}
        = \lrabs{\sum_{i=1}^d\lambda_i^2} 
        \le \sum_{i=1}^d\sigma_0(m)^2 
        = \sigma_0(m)^2 d,
    \end{align*}
    where 
    \[d= \Tr_{S_k^\sigma(N)}T_1'=\frac{k-1}{12}\frac{\psi(N)}{2^{\omega(N)}}+E_{k,\sigma}(1,N)\]
    by Proposition  \ref{prop:final-trace-formula}. 
    Recall by \eqref{eqn:second-coefficient-formula} that
    \[c_2=\frac{1}{2}\left(\left(\Tr_{S_k^\sigma(N)} T'_m\right)^2 - \Tr_{S_k^\sigma(N)} T'^{\,2}_m\right).\]
    Hence using Proposition  \ref{prop:final-trace-formula}, 
    \begin{align} \label{eqn:temp-c2-mainterm-F(m,N,k,sigma)}
        c_2=\frac{1}{2}\lrp{\frac{1}{\sqrt{m}}\frac{k-1}{12}\frac{\psi(N)}{2^{\omega(N)}}}^2+F_{\sigma}(m,N,k)
    \end{align}
    where
    \begin{align*}
        |F_{\sigma}(m,N,k)|\leq& \frac{1}{\sqrt{m}}\frac{k-1}{12}\frac{\psi(N)}{2^{\omega(N)}}|E_{k,\sigma}(m,N)|+\frac{1}{2}|E_{k,\sigma}(m,N)|^2 \\
        &+\frac{1}{2}\frac{k-1}{12}\frac{\psi(N)}{2^{\omega(N)}}\sigma_0(m)^2+\frac{1}{2}|E_{k,\sigma}(1,N)|\sigma_0(m)^2.
    \end{align*}
    Recall that by Proposition \ref{prop:final-trace-formula}, 
    $$|E_{k,\sigma}(m,N)| \le C m^2 \sigma_0(m) \sqrt{N} \log(4N) \sigma_0(N)^2$$
    where $C = 48.99$. Hence, we have
    \begin{align} 
        &\hspace{-3mm} |F_{\sigma}(m,N,k)| \\
        \le
        &\, \frac{1}{\sqrt{m}}\frac{k-1}{12}\frac{\psi(N)}{2^{\omega(N)}}(C  m^2 \sigma_0(m) \sqrt{N} \log(4N) \sigma_0(N)^2) \notag \\ 
        & +\frac{1}{2}\left(C m^2 \sigma_0(m) \sqrt{N} \log(4N) \sigma_0(N)^2\right)^2 \notag \\
        & +\frac{1}{2}\frac{k-1}{12}\frac{\psi(N)}{2^{\omega(N)}}\sigma_0(m)^2+\frac{1}{2}\left(C \sqrt{N} \log(4N) \sigma_0(N)^2\right)\sigma_0(m)^2 \notag \\
        =&\  C \, \frac{k-1}{12}  m^{3/2} \sigma_0(m) \frac{\psi(N)}{2^{\omega(N)}} \sqrt{N} \log(4N) \sigma_0(N)^2 \label{eq:error-term-one} \\
        & + \frac{C^2}{2} m^4 \sigma_0(m)^2 N \log(4N)^2 \sigma_0(N)^4  \label{eq:error-term-two} \\
        & + \frac{1}{2} \frac{k-1}{12} \sigma_0(m)^2  \frac{\psi(N)}{2^{\omega(N)}} \label{eq:error-term-three} \\
        & + \frac{C}{2} \sigma_0(m)^2\sqrt{N} \log(4N) \sigma_0(N)^2. \label{eq:error-term-four} \\
        =&     
        \Bigg[
            \frac{C}{m^{5/2} \sigma_0(m)} 
            +
            \frac{C^2}{2} \frac{12}{k-1} \frac{2^{\omega(N)}}{\psi(N)} \sqrt{N} \log(4N) \sigma_0(N)^2 \\
            &\quad+
            \frac{1}{2} \frac{1}{m^4} \frac{1}{\sqrt{N}\log(4N)\sigma_0(N)^2}
            +
            \frac{C}{2} \frac{12}{k-1} \frac{1}{m^4} \frac{2^{\omega(N)}}{\psi(N)}
        \Bigg] \\
        &\,\times \frac{k-1}{12} m^4\sigma_0(m)^2\frac{\psi(N)}{2^{\omega(N)}}\sqrt{N}\log(4N)\sigma_0(N)^2 \\
        \le&\, \lrb{
            C + \frac{C^2}{2} \cdot 12 \cdot 2240745 + \frac{1}{2\log(4)} + \frac{C}{2} \cdot 12
        } 
        \times \frac{k-1}{12} m^4\sigma_0(m)^2\frac{\psi(N)}{2^{\omega(N)}}\sqrt{N}\log(4N)\sigma_0(N)^2 \\
        \le&\, 3.11 \cdot 10^{10} \, \frac{k-1}{12} m^4\sigma_0(m)^2\frac{\psi(N)}{2^{\omega(N)}}\sqrt{N}\log(4N)\sigma_0(N)^2.
    \end{align}
    
    Note that in the second-to-last step, we used the fact that
    \begin{align}
        \frac{2^{\omega(N)}}{\psi(N)} \sqrt{N} \log(4N) \sigma_0(N)^2  \le 2154495.
    \end{align}
    This follows from the technique of \cite{ross-explicit-bounds}.

    Then observe by \eqref{eqn:temp-c2-mainterm-F(m,N,k,sigma)} that the sign of $c_2$ is $+1$ for all $N,k$ such that 
    $$|F_{\sigma}(m,N,k)|<\frac{1}{2}\lrp{\frac{1}{\sqrt{m}}\frac{k-1}{12}\frac{\psi(N)}{2^{\omega(N)}}}^2.$$
    In particular, this occurs for all $N,k$ such that
    $$3.11 \cdot 10^{10} \, \frac{k-1}{12}m^4 \sigma_0(m)^2\frac{\psi(N)}{2^{\omega(N)}}\sqrt{N}\log(4N)\sigma_0(N)^2<\frac{1}{2}\lrp{\frac{1}{\sqrt{m}}\frac{k-1}{12}\frac{\psi(N)}{2^{\omega(N)}}}^2,$$
    or equivalently
    $$\frac{k-1}{12} \frac{\psi(N)}{2^{\omega(N)} \sqrt{N} \log(4N) \sigma_0(N)^2} > 6.22\cdot 10^{10}\, m^5 \sigma_0(m)^2,$$
    as desired.
    \end{proof}

Next, we consider the case when $m$ is not a square.

\begin{proposition} \label{prop:second-coef-sign-nonsquare}
    Let $m$ be a non-square. Then $c_2(m,N,k,\sigma)$ has sign $-1$ for all $N,k,\sigma$ satisfying 
    \[ \frac{k-1}{12}
    \frac{\psi(N)}{2^{\omega(N)}N(\log 4N)^2\sigma_0(N)^4}
    > 2436\cdot \frac{m^5\sigma_0(m)^2}{\sigma_1(m)}. \]
    
\end{proposition}

\begin{proof}
    First, note that by the Hecke operator composition formula and Proposition  \ref{prop:final-trace-formula}, 
    \begin{align*}
        \Tr_{S^\sigma_k(N)}T_m'^2 &= \sum_{d\mid m}\Tr_{S_k^\sigma (N)}T'_{m^2/d^2}\\
        &=\sum_{d\mid m}\lrp{\frac{d}{m}\frac{k-1}{12}\frac{\psi(N)}{2^{\omega(N)}}+E_{k,\sigma}\lrp{\frac{m^2}{d^2},N}} \\ 
        &=\frac{\sigma_1(m)}{m}\frac{k-1}{12}\frac{\psi(N)}{2^{\omega(N)}}+\sum_{d\mid m}E_{k,\sigma}\lrp{\frac{m^2}{d^2},N}. 
    \end{align*}
    Observe that by Proposition \ref{prop:final-trace-formula},
    \begin{align*}
        \lrabs{\sum_{d\mid m}E_{k,\sigma}\lrp{\frac{m^2}{d^2},N}} 
        &\leq \sum_{d\mid m}C \lrp{\frac{m^2}{d^2}}^2\sigma_0\lrp{\frac{m^2}{d^2}}\sqrt{N}(\log 4N)\sigma_0(N)^2 \\ 
        &\leq C\lrp{\sum_{d\mid m}
        d^4\sigma_0(d^2)}\sqrt{N}(\log 4N)\sigma_0(N)^2 \\ 
        &= C m^4\sigma_0(m)^2 \sqrt{N}(\log 4N)\sigma_0(N)^2, \qquad \text{(by the technique of \cite{ross-explicit-bounds})}
    \end{align*}
    where $C=48.99$
     
    Now, recall that by \eqref{eqn:second-coefficient-formula},
    \[c_2=\frac{1}{2}\left(\left(\Tr_{S_k^\sigma(N)} T'_m\right)^2 - \Tr_{S_k^\sigma(N)} T'^{\,2}_m\right).\]
    Therefore we can write 
    \[c_2=\frac{-1}{2}\frac{\sigma_1(m)}{m}\frac{k-1}{12}\frac{\psi(N)}{2^{\omega(N)}}+G_{\sigma}(m,N,k),\]
    where by Proposition  \ref{prop:final-trace-formula},
    \begin{align} 
      |G_{\sigma}(m,N,k)| &\leq\frac{1}{2} \lrb{ |E_{k,\sigma}(m,N)|^2+
       C m^4\sigma_0(m)^2 \sqrt{N}\log (4N)\sigma_0(N)^2
       } 
      \\
      &\leq \frac{1}{2}\lrb{C^2m^4 \sigma_0(m)^2N\log (4N)^2\sigma_0(N)^4  +C m^4\sigma_0(m)^2 \sqrt{N}\log (4N)\sigma_0(N)^2} 
      \\ 
      &\leq \lrb{\frac{C^2}{2} + \frac{C}{2 \log 4}} m^4 \sigma_0(m)^2 N\log (4N)^2\sigma_0(N)^4 \\
      &\leq 1218\, m^4\sigma_0(m)^2 N\log (4N)^2\sigma_0(N)^4.
      \label{eq:second-coef-error-bound-square-number2}
    \end{align}
    Observe that since the main term of $c_2$ is always negative, $c_2$ will have sign $-1$ for all $k$ satisfying 
    \begin{equation} \label{2nd-coef-condition-nonsquare}
         |G_{\sigma}(m,N,k)|<\frac{1}{2} \frac{\sigma_1(m)}{m} \frac{k-1}{12}\frac{\psi(N)}{2^{\omega(N)}}. 
    \end{equation}
    In particular, by \eqref{eq:second-coef-error-bound-square-number2}, $c_2$ will have sign $-1$ 
    all $k$ such that 
    $$1218\cdot m^4\sigma_0(m)^2N(\log 4N)^2\sigma_0(N)^4 < \frac{1}{2}\frac{\sigma_1(m)}{m} \frac{k-1}{12}\frac{\psi(N)}{2^{\omega(N)}},$$
    or equivalently
    $$\frac{k-1}{12}\frac{\psi(N)}{2^{\omega(N)}N(\log 4N)^2\sigma_0(N)^4} >2436\cdot \frac{m^5\sigma_0(m)^2}{\sigma_1(m)},$$
    as desired.
\end{proof}


\appendix
\section{An explicit trace formula for \texorpdfstring{$T_m'$}{Tm} over \texorpdfstring{$S_k^\sigma(N)$}{Sk\^sigma(N)}}  \label{sec:trace-calculation-appendix}
In this section, we derive estimates for $\Tr_{S_k^\sigma(N)}T_m'$ and $\Tr_{S_k^{\new, \sigma}(N)}T_m'$. We remark that the proof proceeds by a similar strategy as \cite{ross-vanlidth-wolf-xue}. A key difference, however, is that we compute explicit constants for the error bound (needed to make our results effective), whereas \cite{ross-vanlidth-wolf-xue} only gave big-$O$ estimates.

For integers $k\geq 2$, let $p_k(t,m)$ denote the Lucas sequence of the first kind. Namely, $p_k(t,m)$ is the $x^{k-2}$ coefficient in the power series expansion of $(mx^2-tx+1)^{-1}$. Moreover, for integers $N\geq 1$ and $\Delta\leq 0$, define 
\[ H_N(\Delta)=\begin{cases}
    a^2b\lrp{\frac{\Delta/a^2b^2}{N/a^2b}}H(|\Delta/a^2b^2|) &\text{ if } a^2b^2\mid \Delta  \\
    0 &\text{otherwise,}
\end{cases}\]
where $a^2b :=  (N, \Delta)$ with $b$ squarefree, $(\frac{\cdot}{\cdot})$ denotes the Kronecker symbol, and $H(\cdot)$ denotes the Hurwitz-Kronecker class number. In addition, define $B(N)$ to be the greatest integer $d$ such that $d^2\mid N$. We then cite a trace formula for $T_m' \circ W_Q$.

\begin{lemma}[{Skoruppa-Zagier \cite{ZagierSkoruppa1988}, correction by Assaf \cite{assaf2024notetraceformula}}, see also \cite{ross-vanlidth-wolf-xue}] \label{lem:initial-trace-formula}
    For $N \ge 1$ coprime to $m$, $Q \parallel N$, and $k \ge 2$ even, we have
    \begin{align*}
        \Tr_{S_k(N)} T'_m \circ W_Q &= \sum_{\substack{N'\mid N \\ N/N' \text{ squarefree}}} \hspace{-5mm}
        \mu \left( \frac{Q}{(Q, N')} \right) s'_{k, N'}(m, (Q, N')), \\
        \Tr_{S_k^\new(N)} T'_m \circ W_Q &= \sum_{N'\mid N} \alpha \left( \frac{N}{N'} \right) s'_{k, N'}(m, (Q, N')).
    \end{align*}
    Here $\mu$ denotes the Möbius function, $\alpha$ is the multiplicative function defined on prime powers via
    \begin{equation*}
        \alpha \left( p^r \right) = 
        \begin{cases}
            -1 & r = 1 \text{ or } 2 \\
            1 & r = 3 \\
            0 & r \ge 4,
        \end{cases}
    \end{equation*}
    and
    \begin{align}
        s'_{k,N}(m,Q) :=
        -& \frac{1}{2m^{(k-1)/2}} \sum_{Q' \mid Q} \sum_s p_k \left( \frac{s}{\sqrt{Q'}}, m \right) H_{\frac{N}{Q}}\left( s^2 - 4mQ' \right) \label{eq:s-line-1} \\
        -& \frac{1}{2m^{(k-1)/2}} \sum_{m' \mid m} \min \left( m', \frac{m}{m'} \right)^{k-1} \left( B(Q), m + \frac{m}{m'} \right) \left( B \left( \frac{N}{Q} \right), m' - \frac{m}{m'} \right)\quad \label{eq:s-line-2} \\
        +& \frac{1}{m^{(k-1)/2}} \mathbbm{1}_{\substack{k = 2 \\ \frac{N}{Q} = \square}} \sigma_0(Q) \sigma_1(m), \label{eq:s-line-3}
    \end{align}
    with the summation over $s$ ranging over all $s \in \ZZ$ such that $s^2 \le 4mQ'$, $Q' \mid s$, and $\Big( \big( \frac{s}{Q'} \big)^2, \frac{Q}{Q'} \Big)$ is squarefree.
\end{lemma}

We now give an estimate of $s'_{k,N}(m,Q)$ from this trace formula. 


\begin{lemma}[{explicit version of \cite[Lemma 2.2]{ross-vanlidth-wolf-xue}}] \label{lem:s'-estimate}
    Let $m \ge 1$, $N\geq 1$ be coprime to $m$, $Q\parallel N$, and $k\geq 2$ be even. Then 
    \begin{align*}
        s'_{k,N}(m,Q) &= \frac{ \mathbbm{1}_{m=\square}}{\sqrt{m}} 
        \frac{k-1}{12}\frac{N}{Q} 
        -\frac{\mathbbm{1}_{m=\square}}{2}(B(Q),2)B\lrp{\frac{N}{Q}}  \\
        &\quad+ \frac{(-1)^{k/2}}{2\sqrt{m}}
        \sum_{\substack{Q'\mid Q\\ Q/Q' \text{squarefree}}}\hspace{-5mm}
        H_{\frac{N}{Q}}(-4mQ') 
        +E'_{k,Q}(m,N)
    \end{align*}
    where 
    \[|E'_{k,Q}(m,N)| \leq 64.75\cdot m^2\sigma_0(m)\sigma_0(N).\]
\end{lemma}

\begin{proof}
    We evaluate the three terms \eqref{eq:s-line-3}, \eqref{eq:s-line-2}, and \eqref{eq:s-line-1} separately.
    
    First, note that $\sigma_0(Q)\leq \sigma_0(N)$ and $k\geq 2$, and so \eqref{eq:s-line-3} is bounded above by
    \[|\eqref{eq:s-line-3} |\leq \frac{\sigma_0(N)\sigma_1(m)}{\sqrt{m}}.\]

    Second, in the summation over $m'\mid m$ in \eqref{eq:s-line-2}, consider the case where $m'=\sqrt{m}$, which will occur only when $m$ is a square. The contribution of this term is exactly 
    \[
    \eqref{eq:s-line-2}_{m'=\sqrt m} =
    -\frac{\mathbbm{1}_{m=\square}}{2m^{(k-1)/2}}(\sqrt{m})^{k-1}\lrp{B(Q), 2\sqrt{m}}\lrp{B\lrp{\frac{N}{Q}},0}=-\frac{\mathbbm{1}_{m=\square}}{2}\lrp{B(Q),2}B\lrp{\frac{N}{Q}},\]
    since $(Q,m)=1$. In the case where $m' \ne \sqrt m$, note that 
    \begin{align*}
        \left(B(Q),m'+\frac{m}{m'}\right)\leq 2m, \qquad \text{ and } \qquad 
        \left( B\left(\frac{N}{Q}\right),m'-\frac{m}{m'}\right)\leq m.
    \end{align*}
    Thus, \eqref{eq:s-line-2} (excluding the $m'=\sqrt{m}$ case), 
    is bounded above by
    \begin{align*}
        |\eqref{eq:s-line-2}_{m'\ne\sqrt m}|
        \le
        \frac{1}{2m^{(k-1)/2}}\sum_{\substack{m'\mid m\\m'\neq \sqrt{m}}}\min\lrp{m',\frac{m}{m'}}^{k-1}(2m)(m)
        \leq \frac{(2m)(m)}{2}\sum_{m'\mid m}1 = m^2\sigma_0(m).
    \end{align*}
    
    Lastly, in the double summation of  \eqref{eq:s-line-1}, denote $\Delta  :=  s^2-4mQ'$, and break into three cases: when $\Delta < 0, s = 0$; when $\Delta < 0, s \ne 0$; and when $\Delta = 0$.

    \textbf{Case 1: $\Delta < 0, s = 0$. \\}
    In this case, we have that $p_{k}(0,m)= (-1)^{(k-2)/2} m^{(k-2)/2}$.   
    Then since the conditions $0\leq 4mQ'$ and $Q'\mid 0$ are always satisfied, the $\Delta<0, s= 0$ terms of \eqref{eq:s-line-1} are given by
    \begin{align}
        \eqref{eq:s-line-1}_{\Delta<0,s=0}
        &=
        -\frac{1}{2m^{(k-1)/2}}\sum_{\substack{Q'\mid Q\\ Q/Q' \text{squarefree}}} \hspace{-5mm}
        p_k(0,m)H_{\frac{N}{Q}}(-4mQ') \\
        &=\frac{(-1)^{k/2}}{2\sqrt{m}}
        \sum_{\substack{Q'\mid Q\\ Q/Q' \text{squarefree}}}\hspace{-5mm}
        H_{\frac{N}{Q}}(-4mQ').
    \end{align}

    \textbf{Case 2: $\Delta < 0, s \ne 0$. \\}
    In this case, letting $a^2b=\left(\frac{N}{Q}, \Delta \right)$, we have by \cite[proof of Lemma 2.2]{ross-vanlidth-wolf-xue}
    In this case, we have by  \cite[Lemma 2.2]{CLAYTON2024186} that 
    \[\left|p_k\left(\frac{s}{\sqrt{Q'}},m\right)\right|
    \leq \frac{2m^{(k-1)/2}}{\sqrt{\left|\frac{s^2}{Q'}-4m\right|}}
    =\frac{2m^{(k-1)/2}}{\sqrt{|\Delta|/Q'}}.\]
    Using
    \cite[Lemma 2.2]{griffin-ono-tsai}, we can also bound the Hurtwitz-Kronecker class number by 
    \[H(D)\leq \frac{\sqrt{D}(\log D+2)}{\pi}.\]
    Therefore, letting $a^2b=\left(\frac{N}{Q}, \Delta \right)$, we have that
    \begin{align}
        \left|H_{\frac{N}{Q}}(\Delta)\right| &\leq \left|a^2b\left(\frac{\Delta/a^2b^2}{N/a^2b}\right)H(|\Delta/a^2b^2|)\right| \notag\\ 
        &\leq \frac{a^2b\sqrt{|\Delta/a^2b^2|}(\log |\frac{\Delta}{a^2b^2}|+2)}{\pi} \notag \\ 
        &\leq \frac{a\sqrt{|\Delta|}(\log |\Delta|+2)}{\pi}.
        \label{eq:final-classno-bound}
    \end{align}

    Moreover, note that the number of terms in the summation over $s$ (in \eqref{eq:s-line-1}) with $s\neq 0$ is less than or equal to $4\sqrt{m/Q'}$ since $Q'\mid s$ and $|s|\leq 2\sqrt{mQ'}$. For $Q'>16m$, there are no such terms, so it suffices to consider $Q'\leq 16m$. Then for such $Q'$, we have by \eqref{eq:final-classno-bound} that 
    \begin{align*}
        \left|p_k\left(\frac{s}{\sqrt{Q'}},m\right)H_{\frac{N}{Q}}(\Delta)\right|
        &\le \frac{2m^{(k-1)/2}}{\sqrt{\Delta/Q'}}\frac{a\sqrt{|\Delta|}(\log |\Delta |+2)}{\pi}\\
        &=\frac{2m^{(k-1)/2}a\sqrt{Q'}(\log |\Delta|+2)}{\pi} \\
        &\le \frac{2m^{(k-1)/2}(8m)\sqrt{Q'}(\log (64m^2)+2)}{\pi} \\ 
        &=\frac{32}{\pi}m^{(k-1)/2}m\log(8em)\sqrt{Q'}
        ,
    \end{align*}
    where the last inequality comes from the fact that $a^2\leq |\Delta|\leq 4mQ'\leq 64m^2$.

    Now note that the number of terms in the summation over $Q' \mid Q$ is equal to $\sigma_0(Q)\leq \sigma_0(N)$. Thus, the  $\Delta<0, s\neq 0$ terms of \eqref{eq:s-line-1} are bounded by 
    \begin{align}
        |\eqref{eq:s-line-1}_{\Delta<0,s\ne 0}|
        &\le
        \frac{1}{2m^{(k-1)/2}} \cdot \sigma_0(N) \cdot 4\sqrt{\frac{m}{Q'}} \cdot \frac{32}{\pi}m^{(k-1)/2}m\log(8em)\sqrt{Q'} \\
        &= \frac{64}{\pi}m^{3/2}\log(8em)\sigma_0(N).
    \end{align}

    \textbf{Case 3: $\Delta = 0$.\\}
    Finally, we consider the $\Delta=0$ terms, which will make up the main term of $s'_{k,N}(m,Q)$. Recall that $(m,Q')=1$, and so $\Delta=s^2-4mQ'=0$ implies that $m$ and $Q'$ are both squares. Additionally, $(Q')^2\mid s^2=4mQ'$ implies that $Q'\mid 4$. Hence we must have $Q'=1$ or $Q'=4$. It is straightforward to verify that the first case occurs precisely when $m$ is a square and $4\nmid Q$ (for $Q'=1$, $s=\pm 2\sqrt{m}$), and the second case occurs precisely when $m$ is a square and $4\mid Q$ (for $Q'=4$, $s=\pm 4\sqrt{m}$). This means that in all cases, the number of $\Delta =0$ terms appearing in 
    \eqref{eq:s-line-1} is exactly $2\mathbbm{1}_{m=\square}$.  

    Moreover, when $\Delta = 0$, we have by \cite[Lemma 2.3]{CLAYTON2024186} that 
    \[p_k\left(\frac{s}{\sqrt{Q'}},m\right)=(k-1)m^{(k-2)/2},\] and 
    \[H_{\frac{N}{Q}}\left(s^2-4mQ'\right)=H_{\frac{N}{Q}}(0)=\frac{-1}{12}\frac{N}{Q}.\]
    Thus, the $\Delta=0$ terms of \eqref{eq:error-term-one} are given by
    \begin{align*}
    \eqref{eq:s-line-1}_{\Delta=0}
    &=
        -\frac{1}{2m^{(k-1)/2}}\sum_{Q'\mid Q}\sum_{\substack{s^2=4mQ' \\ Q' \mid s}}p_k\left(\frac{s}{\sqrt{Q'}},m\right)H_{\frac{N}{Q}}\left(s^2-4mQ'\right)\\
        &=-\frac{1}{2m^{(k-1)/2}}\left(2\mathbbm{1}_{m=\square}\right)\left((k-1)m^{(k-2)/2}\right)\left(\frac{-1}{12}\frac{N}Q{}\right) \\ 
        &=\frac{\mathbbm{1}_{m=\square}}{\sqrt{m}}\frac{k-1}{12}\frac{N}{Q}. 
    \end{align*}
    
    Collecting all of the above terms, we obtain that 
    \begin{align}
        s'_{k,N}(m,Q)
        &= \eqref{eq:s-line-1}_{\Delta=0}+
         \eqref{eq:s-line-2}_{m'=\sqrt m} + 
         \eqref{eq:s-line-1}_{\Delta<0,s=0} + 
        \eqref{eq:s-line-3} + \eqref{eq:s-line-2}_{m'\ne\sqrt m}  + \eqref{eq:s-line-1}_{\Delta<0,s\ne0}  \\
        &=:  \frac{\mathbbm{1}_{m=\square}}{\sqrt{m}}\frac{k-1}{12}\frac{N}{Q} - \frac{\mathbbm{1}_{m=\square}}{2}(B(Q),2)B\lrp{\frac{N}{Q}}  
        +\frac{(-1)^{k/2}}{2\sqrt{m}}
        \sum_{\substack{Q'\mid Q\\ Q/Q' \text{squarefree}}}\hspace{-5mm}H_{\frac{N}{Q}}(-4mQ') \\
        &\qquad + E'_{k,Q}(m,N),
    \end{align} 
    where,
    \begin{align*}
        |E'_{k,Q}(m,N)| &\leq \frac{\sigma_1(m)}{\sqrt{m}}\sigma_0(N)+m^2\sigma_0(m)+\frac{64}{\pi}m^{3/2}\log(8em)\sigma_0(N)  \\
        &\leq m^2\sigma_0(m)\sigma_0(N)+ m^2\sigma_0(m)\sigma_0(N) + \frac{64}{\pi}(3.08\cdot m^2\sigma_0(m))\sigma_0(N)  \\ 
        & \qquad \text{(since $\log(8em)\leq 3.08\cdot \sqrt{m}$)} \\
        &\leq 64.75\cdot m^2\sigma_0(m)\sigma_0(N),
    \end{align*}
    as desired.
\end{proof}

We can then use the above lemma to estimate $\Tr_{S_k(N)}T_m'\circ W_Q$. 

\begin{lemma}[{explicit version of \cite[Proposition 2.3]{ross-vanlidth-wolf-xue}}] \label{lem:involution-trace-estimate} 
    Let $m \ge 1$, $N \geq 1$ be coprime to $m$, $Q\parallel N$, and $k \geq 2$ be even. Then 
    \begin{align*}
      \Tr_{S_k(N)}T'_m \circ W_Q &= \frac{\mathbbm{1}_{m=\square}}{\sqrt{m}}\frac{k-1}{12}\psi(N) \mathbbm{1}_{Q=1} + E''_{k,Q}(m,N), \\ 
      \Tr_{S_k^{\new}(N)}T'_m\circ W_Q&=\frac{\mathbbm{1}_{m=\square}}{\sqrt{m}}\frac{k-1}{12}\psi^{\new }(N)\eta(Q)+ E''^{\,\new}_{k,Q}({m,N}),
    \end{align*}
    where $\psi$, $\psi^\new$, and $\eta$ are the multiplicative functions defined on prime powers via 
    \begin{align}
        \psi(p^r)&  :=  p^r\left(1+\frac{1}{p}\right), \\ 
        \psi^\new (p^r)&  :=  p^r\lrp{1-\frac{1}{p}
        -\frac{1}{p^2}\mathbbm{1}_{r\geq2}+
        \frac{1}{p^3}\mathbbm{1}_{r\geq 3}},  \label{eq:psi-new-def}\\ 
        \eta(p^r)&  :=  \frac{-\mathbbm{1}_{r=2}}{p^2-p-1}.
    \end{align}
    Note in particular that when $N$ is squarefree, $\eta(Q)=\mathbbm{1}_{Q=1}$. Moreover,
    \begin{align*}
        |E''_{k,Q}(m,N)| &\leq  48.99\cdot m^2\sigma_0(m)\sqrt{N}\log (4N)\sigma_0(N)^2,\\
        |E''^{\,\new}_{k,Q}(m,N)| &\leq 48.99\cdot m^2\sigma_0(m)\sqrt{N}\log (4N)\sigma_0(N)^2.
    \end{align*}
\end{lemma}

\begin{proof}
    First, we have by \cite[Proposition 2.3]{ross-vanlidth-wolf-xue} that
    \begin{equation} \label{eqn:trace-convolution-main-term}
        \sum_{\substack{N' \mid N \\ N/N' \text{squarefree}}} \hspace{-5mm}\mu\lrp{\frac{Q}{(Q,N')}} \frac{N'}{(Q,N')}=\psi(N)\mathbbm{1}_{Q=1}.
    \end{equation}

     Hence by Lemma \ref{lem:initial-trace-formula} and Lemma \ref{lem:s'-estimate}, we have that 
    \begin{align*}
        &\Tr_{S_k(N)}T_m'\circ W_Q \\
        &= \sum_{\substack{N' \mid N \\ N/N' \text{ squarefree}}} \hspace{-5mm} \mu\left(\frac{Q}{(Q,N')}\right)s'_{k,N'}(m,(Q,N')) \\
        &= \sum_{\substack{N'\mid N \\ N/N' \text{ squarefree}}}\hspace{-5mm}\mu\lrp{\frac{Q}{(Q,N')}}\Bigg[\frac{\mathbbm{1}_{m=\square}}{\sqrt{m}}\frac{k-1}{12}\frac{N'}{(Q,N')} -\frac{\mathbbm{1}_{m=\square}}{2}(B((Q,N')),2)B\lrp{\frac{N'}{(Q,N')}} \\
        &\hspace{52mm}+\frac{(-1)^{k/2}}{2\sqrt{m}}\sum_{\substack{Q'\mid (Q,N')\\(Q,N')/Q' \text{ squarefree}}}\hspace{-5mm}H_{\frac{N'}{(Q,N')}}(-4mQ') +E'_{k,(Q,N')}({m,N'})\Bigg] \\
        &= \frac{\mathbbm{1}_{m=\square}}{\sqrt{m}}\frac{k-1}{12}\psi(N)\mathbbm{1}_{Q=1} + 
        \sum_{\substack{N'\mid N \\ N/N' \text{ squarefree}}}\hspace{-5mm}\mu\lrp{\frac{Q}{(Q,N')}}\Bigg[-\frac{\mathbbm{1}_{m=\square}}{2}(B((Q,N')),2)B\lrp{\frac{N'}{(Q,N')}} \\
        &\hspace{52mm}+\frac{(-1)^{k/2}}{2\sqrt{m}}\sum_{\substack{Q'\mid (Q,N')\\(Q,N')/Q' \text{ squarefree}}}\hspace{-5mm}H_{\frac{N'}{(Q,N')}}(-4mQ') +E'_{k,(Q,N')}({m,N'})\Bigg] \\
        &=: \frac{\mathbbm{1}_{m=\square}}{\sqrt{m}}\frac{k-1}{12}\psi(N)\mathbbm{1}_{Q=1} + E''_{k,Q}(m,N). 
    \end{align*}

    We now bound each of the three terms from $E''_{k,Q}(m,N)$ separately. 
    
    First, observe that 
    \begin{align}\label{eq:term2-square-bound}
         \lrabs{-\frac{\mathbbm{1}_{m=\square}}{2}(B((Q,N')),2)B\lrp{\frac{N'}{(Q,N')}} }\leq \frac{1}{2}(2)\lrp{\sqrt{\frac{N'}{(Q,N')}}}\leq \sqrt{N'} \le \sqrt N,
    \end{align}
    so that
    \begin{align} \label{eq:E''-error-term-1}
        &\lrabs{\sum_{\substack{N'\mid N\\N/N' \text{ squarefree}}}
        \hspace{-5mm} \mu \lrp{\frac{Q}{(Q,N')}}\lrp{-\frac{\mathbbm{1}_{m=\square}}{2}(B((Q,N')),2)B\lrp{\frac{N'}{(Q,N')}} } } \leq \sum_{N'\mid N}\sqrt{N}
        \leq \sqrt{N}\sigma_0(N).\qquad
    \end{align}

    Second, let $a^2b=\lrp{\frac{N'}{(Q,N')}, -4mQ'}$ for $Q'\mid Q$, and recall by \eqref{eq:final-classno-bound} that
    \begin{align*}
        \lrabs{H_{\frac{N'}{(Q,N')}}(-4mQ')}&\leq \frac{a\sqrt{4mQ'}(\log(4mQ')+2)}{\pi}
        \\
        &\leq \frac{4\sqrt{mN}\log(4e^2mN)}{\pi},
    \end{align*}
    using the fact that $a\leq 2$ since $\lrp{N',m}=1$ and $\lrp{\frac{N'}{(Q,N')}, Q'}=1$. Therefore,    
    \begin{align}
        \lrabs{\frac{(-1)^{k/2}}{2\sqrt{m}}
        \sum_{\substack{Q'\mid Q\\Q/Q' \text{ squarefree}}}\hspace{-5mm}H_{\frac{N}{Q}}(-4mQ')} 
        &\leq \frac{1}{2\sqrt{m}}
        \sum_{Q'\mid Q}\frac{ 4\sqrt{mN}\log(4e^2mN)}{\pi} \notag \\ 
        &\leq \frac{2\sqrt{N}\log(4e^2mN)}{\pi}\sigma_0(Q)  \notag \\
        &\leq  \frac{2\sqrt{N}\sigma_0(N)\log(4e^2mN)
        }{\pi},  \label{term3-s=0-bound}
    \end{align}
    so that
    \begin{align}
    &\lrabs{\sum_{\substack{N'\mid N \\ N/N' \text{ squarefree}}}\hspace{-5mm}
    \mu\lrp{\frac{Q}{(Q,N')}}\frac{(-1)^{k/2}}{2\sqrt{m}}
    \hspace{-3mm}\sum_{\substack{Q'\mid (Q,N')\\(Q,N')/Q' \text{ squarefree}}} \hspace{-5mm} H_{\frac{N'}{(Q,N')}}(-4mQ') } \\
    &\leq \sum_{\substack{N'\mid N\\N/N' \text{ squarefree}}}\hspace{-5mm}\frac{2\sqrt{N'}\sigma_0(N')^2\log(4e^2mN')}{\pi} \\
    &\leq  \quad \frac{2}{\pi} \sqrt{N}\sigma_0(N)^2\log(4e^2mN). 
    \label{eq:E''-error-term-2}
    \end{align}

    Third, note that  
    \begin{align} 
        \lrabs{\sum_{\substack{N'\mid N\\ N/N' \text{squarefree}}}\hspace{-5mm}\mu \lrp{\frac{Q}{(Q,N')}}E'_{k,(Q,N')}({m,N')}}
        &\le \sum_{\substack{N'\mid N\\ N/N' \text{squarefree}}} \lrabs{\mu\lrp{\frac{Q}{(Q,N')}}E'_{k,(Q,N')}(m,N')} \\ 
        &\le\sum_{N'\mid N}
        |E'_{k,(Q,N')}(m,N')| \label{eq:E''-error-term-3}
    \end{align}

    Combining these three terms 
    \eqref{eq:E''-error-term-3},
    \eqref{eq:E''-error-term-1}, and \eqref{eq:E''-error-term-2} (and utilizing the bound for $E'_{k,Q}(m,N)$ from Lemma \ref{lem:s'-estimate}), we obtain that
    \begin{align*}
        &|E''_{k,Q}(m,N)| \\
        &\leq \sum_{N'\mid N}|E'_{k,(Q,N')}(m,N')|+
        \sqrt{N}\sigma_0(N)+\frac{2}{\pi}\sqrt{N}\sigma_0(N)^2\log(4e^2mN) \\ 
        &\leq 64.75\cdot m^2\sigma_0(m)\sigma_0(N)^2+
        \sqrt{N}\sigma_0(N)  +\frac{2}{\pi}\sqrt{N}\log(4N)\sigma_0(N)^2 +\frac{2}{\pi}\log(e^2m)\sqrt{N}\sigma_0(N)^2 \\ 
        &= \lrb{64.75\frac{1}{\sqrt{N}\log(4N)} + \frac{1}{m^2\sigma_0(m)}\frac{1}{\log(4N)\sigma_0(N)} + \frac{2}{\pi}\frac{1}{m^2\sigma_0(m)} + \frac{2}{\pi}\frac{\log(e^2m)}{m^2\sigma_0(m)}\frac{1}{\log(4N)}} \\
        &\quad \times m^2\sigma_0(m)\sqrt{N}\log(4N)\sigma_0(N)^2 \\ 
        &\leq 
        \lrb{
            64.75 \cdot \frac{1}{\log 4} +  1 \cdot \frac{1}{\log 4} + \frac{2}{\pi} +
            \frac{2}{\pi} \cdot 2 \cdot \frac{1}{\log 4}
        }
        \cdot  m^2\sigma_0(m)\sqrt{N}\log(4N)\sigma_0(N)^2 \\
        &\leq 48.99\cdot  m^2\sigma_0(m)\sqrt{N}\log(4N)\sigma_0(N)^2,
    \end{align*}
    as desired.

    To show the desired identity for the newspace, apply Lemmas \ref{lem:initial-trace-formula} and \ref{lem:s'-estimate} in a similar manner, and note that all error terms have bounds identical to the ones in the fullspace identity, and we bound both $\mu$ and $\alpha$ above by $|\mu(\cdot)|,|\alpha (\cdot)|\leq 1$. See \cite[Proposition 3.2]{ross-vanlidth-wolf-xue} for the calculation of the main term $\psi^\nw\!(N) \eta(Q)$. 
\end{proof}

Now, 
$\Tr_{S_k^\sigma(N)}T'_m$ and  $\Tr_{S_k^{\new,\sigma}(N)}T'_m$ can be expressed in terms of $\Tr_{S_k(N)}T_m'\circ W_Q$ and $\Tr_{S_k^{\new}(N)}T_m'\circ W_Q$ respectively:

\begin{lemma}[{\cite[Lemma 3.1]{ross-vanlidth-wolf-xue}}] \label{lem:involution-to-sign-trace}
    For $N\geq 1$ and a sign pattern $\sigma$ for $N$, 
    \begin{align*}
        \Tr_{S_k^\sigma(N)}T'_m &= \frac{1}{2^{\omega(N)}}\sum_{Q||N}\sigma(Q)\Tr_{S_k(N)}T'_m\circ W_Q, \\ 
        \Tr_{S_k^{\new, \sigma}(N)}T'_m &= \frac{1}{2^{\omega(N)}}\sum_{Q||N}\sigma(Q)\Tr_{S_k^\new(N)}T'_m\circ W_Q.
    \end{align*}
\end{lemma}

Lemmas \ref{lem:involution-trace-estimate} and \ref{lem:involution-to-sign-trace} then yield the following trace estimate.


\begin{proposition}[{explicit version of \cite[Corollary 3.2]{ross-vanlidth-wolf-xue}}] \label{prop:final-trace-formula}
    For $m \ge 1$, $N \ge 1$ coprime to $m$, sign patterns $\sigma$ for $N$, and $k \ge 2$ even,
    \begin{align*}
        \Tr_{S_k^\sigma(N)}T'_m &= \frac{\mathbbm{1}_{m=\square}}{\sqrt{m}}\frac{k-1}{12}\frac{\psi(N)}{2^{\omega(N)}} +E_{k,\sigma}(m,N), \\
        \Tr_{S_k^{\new,\sigma}(N)}T'_m &= \frac{\mathbbm{1}_{m=\square}}{\sqrt{m}}\frac{k-1}{12}\frac{\psi^\new(N)}{2^{\omega(N)}} \prod_{p^r||N} \left( 1 + \sigma\left(p^r\right) \frac{-\mathbbm{1}_{r=2}}{p^2-p-1} \right) +E^{\new}_{k,\sigma}(m,N),
    \end{align*}
    where 
    \begin{align*}
        |E_{k,\sigma}({m,N})| &\le    48.99\cdot m^2\sigma_0(m)\sqrt{N}\log (4N)\sigma_0(N)^2, \\ 
        |E^\new_{k,\sigma}({m,N})| &\le   48.99\cdot m^2\sigma_0(m)\sqrt{N}\log (4N)\sigma_0(N)^2.
    \end{align*}
\end{proposition}

\begin{proof}
    With Lemma \ref{lem:involution-to-sign-trace} and Lemma \ref{lem:involution-trace-estimate}, we have that
    \begin{align*}
 \Tr_{S_k^\sigma(N)}T'_m &= \frac{1}{2^{\omega(N)}} \sum_{Q||N} \sigma(Q) \Tr_{S_k(N)}T'_m \circ W_Q \\
        &= \frac{1}{2^{\omega(N)}} \sum_{Q||N} \sigma(Q)\lrp{\frac{\mathbbm{1}_{m=\square}}{\sqrt{m}}\frac{k-1}{12}\psi(N)\mathbbm{1}_{Q=1} +E_{k,Q}''(m,N)} \\
        &= \frac{\mathbbm{1}_{m=\square}}{\sqrt{m}}\frac{k-1}{12}\frac{\psi(N)}{2^{\omega(N)}} +\frac{1}{2^{\omega(N)}}\sum_{Q||N}\sigma(Q)E_{k,Q}''(m,N) \\
        &=: \frac{\mathbbm{1}_{m=\square}}{\sqrt{m}}\frac{k-1}{12}\frac{\psi(N)}{2^{\omega(N)}}
        +
        E_{k,\sigma}(m,N),
    \end{align*}
    where
    \begin{align*}
        |E_{k,\sigma}(m,N)|
        &\leq \frac{1}{2^{\omega(N)}}\sum_{Q\parallel N}\lrabs{E''_{k,Q}(m,N)} \\
        &\leq  \frac{1}{2^{\omega(N)}}\sum_{Q \parallel N} 48.99\cdot  m^2\sigma_0(m)\sqrt{N}\log(4N)\sigma_0(N)^2 \\ 
        &=48.99\cdot  m^2\sigma_0(m)\sqrt{N}\log(4N)\sigma_0(N)^2,
    \end{align*}
    verifying the desired result for the fullspace. The argument for the newspace is identical.
\end{proof}

\section{A specialization of the trace formula to prime level} \label{sec:trace-formula-prime}

Observe that the error term in the above trace formula (Proposition \ref{prop:final-trace-formula}) for $T_m'$ over $S_k^\sigma(N)$ 
is of the form $O_m(N^{1/2+\varepsilon})$. This is in stark constrast to the trace formula for $T_m'$ over $S_k(N)$, which has an error term of size $O(N^{\varepsilon})$ when $m$ is non-square \cite[Lemma 2.2]{ross-xue}.
Moreover, we remark that the exponent of $N$ in our error term $O_m(N^{1/2+\varepsilon})$ cannot be improved to $\frac{1}{2}$. Specifically, when $m$ is non-square and $N=p$ is prime, the trace formula has a term proportional to the class number $H(4mp)$ which is $\gg_m p^{1/2} \log \log p$ for a certain family of primes $p$. We use this fact in Section \ref{sec:counterexample} to show that an analogous result to Corollary \ref{cor:even-coef-sign-nonsquare} does not hold asymptotically in $N$.

To this end, we specialize the trace formula from Proposition \ref{prop:final-trace-formula} to the case where $N=p$ is prime.

\begin{lemma} \label{lem:prime-involution-trace-estimate}
    Let $m \geq 1$, $p$ be a prime not dividing $m$, and $k \geq 2$ be even. Then 
    \begin{align*}
        \Tr_{S_k(p)}T_m'\circ W_1 &= \frac{\mathbbm{1}_{m=\square}}{\sqrt{m}}\frac{k-1}{12}(p+1)+O_m(1), \\
        \Tr_{S_k^\new(p)}T_m'\circ W_1 &=\frac{\mathbbm{1}_{m=\square}}{\sqrt{m}} \frac{k-1}{12} (p-1) +O_m(1),
    \end{align*}
    and
    \begin{align*}
        \Tr_{S_k(p)}T'_m \circ W_p &= \frac{(-1)^{k/2}}{2\sqrt{m}} H(4mp) + O_m(1), \\
        \Tr_{S_k^\new(p)}T'_m \circ W_p &= \frac{(-1)^{k/2}}{2\sqrt{m}} H(4mp) + O_m(1).
    \end{align*}
\end{lemma}

\begin{proof}
    First, we note that by Lemma \ref{lem:s'-estimate}, 
    \begin{align*}
        s'_{k,1}(m,1) &= \frac{\mathbbm{1}_{m=\square}}{\sqrt{m}} 
        \frac{k-1}{12} - \frac{\mathbbm{1}_{m=\square}}{2} + \frac{(-1)^{k/2}}{2\sqrt{m}}H_{1}(-4m) + E'_{k,1}(m,1) \\ 
        &=\frac{\mathbbm{1}_{m=\square}}{\sqrt{m}} 
        \frac{k-1}{12} +O_m(1),
    \end{align*}
    and 
    \begin{align*}
        s'_{k,p}(m,p) &= \frac{\mathbbm{1}_{m=\square}}{\sqrt{m}} 
        \frac{k-1}{12} - \frac{\mathbbm{1}_{m=\square}}{2} 
        +\frac{(-1)^{k/2}}{2\sqrt{m}}(H_1(-4m)+H_1(-4mp)) + E'_{k,p}(m,p) \\ 
        &=\frac{\mathbbm{1}_{m=\square}}{\sqrt{m}} 
        \frac{k-1}{12} + \frac{(-1)^{k/2}}{2\sqrt{m}}H(4mp)+ O_{m}(1),
    \end{align*}
    since $E'_{k,p}(m,p)=O_m(\sigma_0(p))=O_m(1)$. Hence by Lemma \ref{lem:initial-trace-formula} ,
    \begin{align*}
      \Tr_{S_k(p)}T_m'\circ W_p 
      =& \,\mu\lrp{\frac{p}{(p,1)}}s'_{k,1}(m,(p,1))
         +\mu\lrp{\frac{p}{(p,p)}}s'_{k,p}(m,(p,p))   \\
      =& -s'_{k,1}(m,1)+s'_{k,p}(m,p) \\
      =& \,\frac{(-1)^{k/2}}{2\sqrt{m}}H(4mp) + O_m(1),
    \end{align*}
    as desired.
    
    Additionally, using Lemma \ref{lem:s'-estimate}, 
    \begin{align*}
        s'_{k,p}(m,1) &= \frac{\mathbbm{1}_{m=\square}}{\sqrt{m}} \frac{k-1}{12} p
            - \frac{\mathbbm{1}_{m=\square}}{2}
            + \frac{(-1)^{k/2}}{2\sqrt{m}}H_{p}(-4m) + E'_{k,1}(m,p) \\
            &=\frac{\mathbbm{1}_{m=\square}}{\sqrt{m}} \frac{k-1}{12} p+O_m(1),
    \end{align*}
    so that 
    \begin{align*}
        \Tr_{S_k(p)}T_m'\circ W_1 &= 
         \mu\lrp{\frac{1}{(1,1)}}s'_{k,1}(m,(1,1))+
         \mu\lrp{\frac{1}{(1,p)}}s'_{k,p}(m,(1,p)) \\
         &=s'_{k,1}(m,1)+s'_{k,p}(m,1)  \\
        &=\frac{\mathbbm{1}_{m=\square}}{\sqrt{m}} \frac{k-1}{12} (p+1)
        + O_m(1)
    \end{align*}
    as desired. The proof for the newspace is nearly identical.
\end{proof}

\begin{proposition} \label{prop:prime-trace-formula-appendix}
    Let $m\geq 2$, $p$ be a prime not dividing $m$, and $k\geq 2$ be even. Then 
    \begin{align*}
      \Tr_{S_k^\pm(p)}T_m' &= \frac{\mathbbm{1}_{m=\square}}{\sqrt{m}}\frac{k-1}{24}(p+1) \pm \frac{(-1)^{k/2}}{4\sqrt{m}}H(4mp)+O_m(1), \\ 
      \Tr_{S_k^{\new,\pm}(p)}T_m' &= \frac{\mathbbm{1}_{m=\square}}{\sqrt{m}}\frac{k-1}{24}(p-1) \pm \frac{(-1)^{k/2}}{4\sqrt{m}}H(4mp)+O_m(1). 
    \end{align*}
\end{proposition}
\begin{proof}
    By Lemma \ref{lem:involution-to-sign-trace} and Lemma \ref{lem:prime-involution-trace-estimate}, we have
    \begin{align*}
        \Tr_{S_k^\pm(p)}T_m'
        &=\frac{1}{2^{\omega(p)}}\lrp{\Tr_{S_k(p)}T_m'\circ W_1\pm\Tr_{S_k(p)}T_m'\circ W_p} \\
        &=\frac{1}{2}\lrp{\Tr_{S_k(p)}T_m'\circ W_1\pm\Tr_{S_k(p)}T_m'\circ W_p} \\ 
        &=\frac{\mathbbm{1}_{m=\square}}{\sqrt{m}}\frac{k-1}{24}(p+1)\pm \frac{(-1)^{k/2}}{4\sqrt{m}}H(4mp)+O_m(1)
    \end{align*}
    as desired. The proof for the newspace is identical.
\end{proof}


\section{Asymptotic behavior of a family of class numbers} \label{sec:class_number_lemma}

The goal of this appendix is to prove the following proposition. The idea behind this kind of result is standard; for example see \cite{bateman-chowla-erdos}. However, we are not aware of any reference in the literature to the result itself, so we write down the details of the proof.
\begin{proposition}\label{prop:counterexample-prop}
    Fix $m\geq 1$. Then there exists an infinite family of primes $p$ such that 
    \[ H(4mp) \gg_m \sqrt{p} \log\log p,  \]
    and an infinite family of primes $p$ such that
    \[ H(4mp) \ll_m   \frac{\sqrt{p}}{\log \log p}. \]
\end{proposition}

The following lemma reduces the proof of Proposition \ref{prop:counterexample-prop} to estimating the $L$-values $L(1,\chi)$. This result comes from \cite[second displayed equation of the paper]{williams} (which relates $H(4mp)$ to $h(\Delta p)$) and the Dirichlet class number formula (which relates $h(\Delta p)$ to $L(1, \chi_{\Delta p})$).
\begin{lemma} \label{lem:H-classno_to_L-1-chi}
    Fix $m \ge 1$ and write $-4m = \Delta F^2$, where $\Delta$ is a negative fundamental discriminant. Then for primes $p \equiv 1 \pmod{4m}$, we have
    \begin{align}
        H(4mp) \asymp_m h(\Delta p) \asymp_m \sqrt p\, L(1, \chi_{\Delta p}).
    \end{align}
    Note that the characters $\chi_{\Delta p}$ here are primitive.
\end{lemma}

Next, Lemma \ref{lem:estimate-L-1-chi} shows that most values of $L(1,\chi)$ can be approximated by the first few terms of the corresponding Euler product. This result comes from the second half of \cite[Proposition 2.2]{Granville-Sound} at $A=14$. Here, and throughout the rest of this section, a product over $\ell$ means that the product is taken over primes $\ell$.
\begin{lemma}
    \label{lem:estimate-L-1-chi}
    As $Q \to \infty$, we have that
    \begin{align}
        L(1,\chi) \sim \prod_{\ell \le (\log Q)^{14}} \lrp{1 - \frac{\chi(\ell)}{\ell}}^{-1} 
    \end{align}
    for all but at most $Q^{2/13}$ primitive characters of conductor $\le Q$.
\end{lemma}

In light of Lemma \ref{lem:H-classno_to_L-1-chi}, the following lemma then implies the desired result Proposition \ref{prop:counterexample-prop}.
\begin{lemma}
    Fix $m \ge 1$ and write $-4m = \Delta F^2$, as in Lemma \ref{lem:H-classno_to_L-1-chi}. 
    Then there exist infinitely many primes $p \equiv 1 \pmod{4m}$ such that
    \begin{align}
        L(1, \chi_{\Delta p}) \gg_m \log \log p,
    \end{align}
    and infinitely many primes $p \equiv 1 \pmod{4m}$ such that
    \begin{align}
        L(1, \chi_{\Delta p}) \ll_m \frac{1}{\log \log p}.
    \end{align}
\end{lemma}
\begin{proof}
    Fix $\epsilon = \pm 1$. Then it suffices to show that for sufficiently large $x$, there exists a prime $p_x \equiv 1 \pmod{4m}$ satisfying
    \begin{align}
        L(1, \chi_{\Delta p_x})^\epsilon \gg_m \log x \gg_m \log \log p_x. \label{eqn:temp-goal-L1chi}
    \end{align}
    
    For large $x$, let 
    \begin{align}
        M_x := 4m \prod_{\ell \le x,\ \ell \nmid 4m} \ell,
    \end{align}
    and note that 
    \begin{align} \label{eqn:temp-growth-log-Mx}
        \log M_x &= \log 4m + \sum_{\ell \le x,\ \ell \nmid 4m} \log \ell \ \sim\   x \qquad \text{(by PNT)}.
    \end{align}
    Then consider the reduced equivalence classes $r \!\pmod{M_x}$ satisfying the congruence conditions
    \begin{align} \label{eqn:residue-equivalences}
        r \equiv 1 \pmod{4m} \qquad \text{and} \qquad \lrp{\frac{r}{\ell}} &= \epsilon \chi_\Delta(\ell) \quad\text{for } \ell \le x, ~ \ell \nmid 4m. 
    \end{align}
    For each of these reduced equivalence classes $r \!\pmod{M_x}$, choose the smallest prime $p$ such that $p \equiv r \pmod{M_x}$, and let $P_x$ denote the set of such primes. Note that the size of $P_x$ (i.e. the number of such reduced residue classes $r \!\pmod{M_x}$) is given by $\# P_x = \prod_{\ell \le x,\ \ell \nmid 4m} \frac{\ell-1}{2}$, since $\chi_\Delta(\ell)= \pm 1$ for each $\ell \nmid 4m$. Moreover, for sufficiently large $x$, we have by Linnik's theorem that $p \le M_x^6$ for each $p \in P_x$ \cite{heath-brown}.
    
    These primes  $p \in P_x$ then yield $\# P_x$ primitive characters of the form $\chi_{\Delta p}$, each of conductor $\mathrm{cond}(\chi_{\Delta p}) = |\Delta|p \le Q_x := |\Delta| M_x^6$. Also, note that $\# P_x > Q_x^{2/13}$ for sufficiently large $x$ since
    \begin{align}
        \log \# P_x &= \log \prod_{\ell \le x,\ \ell \nmid 4m} \frac{\ell-1}{2} =\sum_{\ell \le x,\ \ell \nmid 4m} \lrp{\log \ell+ \log \frac{\ell-1}{2\ell}} \ \sim\  x \qquad \text{(by PNT)},
        \\
        \log Q_x^{2/13} 
        &= \frac{2}{13} \log |\Delta| + \frac{12}{13} \log M_x \ \sim \ \frac{12}{13} x \qquad \text{(by \eqref{eqn:temp-growth-log-Mx})}. \label{eqn:growth-of-Q^2/13}
    \end{align}
    Hence by Lemma \ref{lem:estimate-L-1-chi}, there must exist at least one prime in $P_x$, say $p_x$, such that 
    \begin{align}
        L(1, \chi_{\Delta p_x}) \sim \prod_{\ell \le (\log Q_x)^{14}} \lrp{1 - \frac{\chi_{\Delta p_x}(\ell)}{\ell}}^{-1}.
    \end{align}

    Now, note from \eqref{eqn:growth-of-Q^2/13} that $(\log Q_x)^{14} > x$ for sufficiently large $x$. Hence
    \begin{align}
        L(1,\chi_{\Delta p_x})^\epsilon 
        &\sim \prod_{\ell \le (\log Q_x)^{14}} \lrp{1 - \frac{\chi_{\Delta p_x}(\ell)}{\ell}}^{-\epsilon} \\
        &= 
        \prod_{\ell \le x} \lrp{1 - \frac{\chi_{\Delta p_x}(\ell)}{\ell}}^{-\epsilon} 
        \cdot 
        \prod_{x <\ell \le (\log Q_x)^{14}} \lrp{1 - \frac{\chi_{\Delta p_x}(\ell)}{\ell}}^{-\epsilon} 
        \hspace{-30mm}\\
        &\gg_m \prod_{\ell \le x} \lrp{1 - \frac{\epsilon}{\ell}}^{-\epsilon}
        \cdot 
        \prod_{x <\ell \le (\log Q_x)^{14}} \lrp{1 - \frac{1}{\ell}} \\
        & &\hspace{-70mm}
        \Big(\text{since $\chi_{\Delta p_x}(\ell) = \Big(\frac{\Delta}{\ell}\Big) \cdot \Big(\frac{p_x}{\ell}\Big) = \epsilon$ for $\ell \le x, \ell \nmid 4m$ by \eqref{eqn:residue-equivalences}}\Big)
        \\
        &\asymp_m \log x \cdot \frac{\log x}{\log\!\big((\log Q_x)^{14}\big)}
        &\hspace{-20mm}\text{(by Mertens' third theorem)}
        \\
        &\asymp_m \log x 
        &\text{(by \eqref{eqn:growth-of-Q^2/13})} \\
        &\asymp_m \log \log M_x^6 
        &\text{(by \eqref{eqn:temp-growth-log-Mx})} \\
        &\ge \log \log p_x,
    \end{align}
    verifying \eqref{eqn:temp-goal-L1chi}. This completes the proof.
\end{proof}


\bibliographystyle{plain}
\bibliography{bibliography.bib}

\end{document}